\documentclass[english]{article}
\usepackage[hyperindex,breaklinks,colorlinks,citecolor=blue,urlcolor=blue]{hyperref}
\usepackage[T1]{fontenc}
\usepackage{geometry}
\usepackage{float}
\usepackage{mathrsfs}
\usepackage{amsmath}
\usepackage{amsfonts}
\numberwithin{equation}{section}
\usepackage{amssymb}
\usepackage{amsthm}
\usepackage{graphicx}
\usepackage{caption}
\usepackage{subcaption}
\usepackage{bbm}
\usepackage{bm}

\makeatletter

\floatstyle{ruled}
\newfloat{algorithm}{tbp}{loa}
\providecommand{\algorithmname}{Algorithm}
\floatname{algorithm}{\protect\algorithmname}

\usepackage{amsthm}

\usepackage{mathrsfs}

\usepackage{amsfonts}

\usepackage{epsfig}

\usepackage{bm}

\usepackage{mathrsfs}

\usepackage{enumerate}

\@ifundefined{definecolor}{\@ifundefined{definecolor}
 {\@ifundefined{definecolor}
 {\usepackage{color}}{}
}{}
}{}
\usepackage[all]{xy}

\newtheorem{theorem}{Theorem}[section]
\newtheorem{lem}{Lemma}[section]
\newtheorem{rem}{Remark}[section]
\newtheorem{prop}{Proposition}[section]

\newtheorem{ass}{Assumption}[section]

\newcounter{hypA}

\usepackage{babel}\date{}

\usepackage{babel}

\makeatother

\usepackage{babel}

\newcommand{\D}{\bm{D}}

\newcommand{\HH}{\mathcal{H}}

\newcommand{\U}{\mathcal{U}}

\newcommand{\E}{\mathbb{E}}

\newcommand{\PP}{\mathbb{P}}

\DeclareMathOperator*{\argmin}{arg\,min}

\newcommand{\II}{\mathbbm{1}}
\newcommand{\ol}{\overline}

\newcommand{\R}{\mathbb{R}}

\newcommand{\Z}{\mathbb{Z}}

\newcommand{\N}{\mathbb{N}}

\newcommand{\sgn}{\mathrm{sgn}}

\newcommand{\vct}[1]{\bm{#1}}
\newcommand{\mtx}[1]{\bm{#1}}

\newcommand{\grad}{\nabla}

\newcommand{\diam}{\mathrm{diam}}

\renewcommand{\l}{\left}
\renewcommand{\r}{\right}

\newcommand{\musm}{\mu^{\mathrm{sm}}}

\newcommand{\mufi}{\mu^{\mathrm{fi}}}

\newcommand{\Cder}{C_{\mathrm{der}}}
\newcommand{\CJ}{C_{\mtx{J}}}
\newcommand{\BR}{\mathcal{B}_R}

\newcommand{\mr}{\mathrm}

\renewcommand{\phi}{\varphi}
\renewcommand{\epsilon}{\varepsilon}

\newcommand{\vmaxgeo}{v_{\mathrm{max}}^{\mathrm{geo}}}
\newcommand{\vmingeo}{v_{\mathrm{min}}^{\mathrm{geo}}}

\newcommand{\supp}{\mathrm{supp}}

\renewcommand{\u}{\vct{u}}
\renewcommand{\v}{\vct{v}}
\newcommand{\w}{\vct{w}}
\newcommand{\J}{\mtx{J}}
\newcommand{\tU}{\tilde{U}}
\newcommand{\TT}{\mathcal{T}}
\newcommand{\Lambdageo}{\Lambda^{\mathrm{geo}}}
\renewcommand{\H}{\mtx{H}}

\begin{document}

\begin{center}
{\Large \textbf{On Concentration Properties of Partially Observed Chaotic Systems}}

\vspace{0.5cm}

BY DANIEL PAULIN$^{1}$, AJAY JASRA$^{1}$,  DAN CRISAN$^{2}$ \& ALEXANDROS BESKOS$^{3}$

{\footnotesize $^{1}$Department of Statistics \& Applied Probability,
National University of Singapore, Singapore, 117546, SG.}\\
{\footnotesize E-Mail:\,}\texttt{\emph{\footnotesize paulindani@gmail.com, staja@nus.edu.sg}}\\
{\footnotesize $^{2}$Department of Mathematics,
Imperial College London, London, SW7 2AZ, UK.}\\
{\footnotesize E-Mail:\,}\texttt{\emph{\footnotesize d.crisan@ic.ac.uk}}\\
{\footnotesize $^{3}$Department of Statistical Science,
University College London, London, WC1E 6BT, UK.}\\
{\footnotesize E-Mail:\,}\texttt{\emph{\footnotesize a.beskos@ucl.ac.uk}}
\end{center}

\begin{abstract}
This article 
presents results on 
the concentration properties of the smoothing and filtering distributions of some partially observed chaotic dynamical systems.
We show that, rather surprisingly, for the geometric model of the Lorenz equations, as well as some other chaotic dynamical systems, the smoothing and filtering distributions do not concentrate around the true position of the signal, as the number of observations tends to infinity. 
Instead, under various assumptions on the observation noise, we show that the expected value of the diameter of the support of the smoothing and filtering distributions remains \emph{lower} bounded by a constant times the standard deviation of the noise, independently  of the number of observations. Conversely, under rather general conditions, 
the diameter of the support of the smoothing and filtering distributions are \emph{upper} bounded by a constant times the standard deviation of the noise. 
To some extent, applications to the three dimensional Lorenz 63' model and to the Lorenz
96' model of arbitrarily large dimension are considered.
\end{abstract}
\section{Introduction}

The filtering and smoothing problems are ubiquitous in many areas, such
as statistics, engineering, econometrics and meteorology; see for instance
\cite{dan} and the references therein. Such problems are concerned with inference of the current (filtering) or past (smoothing) positions of a partially observed dynamical system conditional
upon sequentially observed data. Perhaps the most well-studied class of filtering
and smoothing problems are those for which the unobserved signal follows a Markov chain in discrete time, and the observations
at the current time are, conditional upon the signal at the current time, independent
of all other random variables. This is the so-called state-space or hidden Markov model; see for instance \cite{cappe} for a book length introduction. For the aforementioned models, a wealth of results on long-time behaviour 
and concentration of the system exist; see for instance \cite{cappe,del2013mean,vanhandel}. Potentially less studied
in the literature are such results for the case for which the unobserved system is deterministic, with unknown initial condition (see \cite{lawstuart} for examples of this type of models). Such models have a wide class of applications, for instance,
in weather prediction (especially when the dynamics are chaotic), but there are relatively few mathematical results on the concentration of the smoother
and filter on the true position; see  \cite{pires1996extending, CerouLongtimebehaviour,AlonsoStuartLongtime, Lalleybeneaththenoise, LalleyNobel}. We include a detailed comparison with the latter two papers in Section \ref{seccomparisonlalley}. We note that \cite{pires1996extending} has studied this problem from a practical perspective, but the key statement (iii) in Section 3.3 is only applicable to uniformly hyperbolic systems, which excludes most practically relevant models (such as the Lorenz 63' model). The concentration properties of the smoother and the filter are important particularly when assessing the ability to fit such models to data.

In this paper we investigate the behaviour of the smoothing and filtering distributions of partially observed deterministic dynamical systems of the general form
\begin{equation}\label{diffeqgeneralform}
\frac{d \u }{d t}=-\mtx{A} \u -\mtx{B}(\u,\u)+\vct{f},
\end{equation} 
where $\u: \R^+\to \HH$ is a dynamical system in a Hilbert space $\HH$, $\mtx{A}$ is a linear operator on $\HH$, $\vct{f}\in \HH$ is a constant vector, and $\mtx{B}(\u,\u)$ is a bilinear form corresponding to the nonlinearity. In this paper we will work with finite dimensional systems, thus we assume that
\begin{equation}\label{Hdef}
\HH:=\R^d \text { for some }d\in \Z_+.
\end{equation} 
This is required due to the fact that in general it is not easy to obtain precise distributional information about an infinite dimensional system based on finite dimensional observations (unless only a finite dimensional part of the system is important, and the rest is negligible).

For $t\ge 0$, let $\v(t)$ denote the solution of \eqref{diffeqgeneralform} started from some $\v\in \R^d$. This can be shown to exist locally. The derivatives of the solution $\v(t)$ at time $t=0$ will be denoted by 
\begin{equation}
\D^i \v:=\l.\frac{d^i \v(t)}{dt^i}\r|_{t=0} \text{ for }i\in \N,
\end{equation}
in particular, $\D^0\v=\v$ , $\D\v:=\D^1\v=-\mtx{A} \v -\mtx{B}(\v,\v)+\vct{f}$ (the right hand side of \eqref{diffeqgeneralform}), and $\D^2\v=-\mtx{A} \D^1 \v-\mtx{B}(\D^1 \v,\v)-\mtx{B}(\v,\D^1 \v)$.

In order to ensure the existence of a solution to the equation \eqref{diffeqgeneralform} for every $t\ge 0$, we assume that there are constants $R>0$ and $\delta>0$ such that 
\begin{equation}\label{eqtrappingball}
\l<\D\v,\v\r> \le 0 \text{ for every }\v\in \R^d \text{ with }\|\v\|\in [R,R+\delta].
\end{equation}
We call this the \emph{trapping ball} assumption. Let $\BR:=\{\v\in \R^d: \|\v\|\le R\}$ be the ball of radius $R$. Using the fact that $\l<\frac{d}{dt} \v(t),\v(t)\r>= \frac{1}{2}\frac{d }{d t}\|\v(t)\|^2$, one can show that the solution to \eqref{diffeqgeneralform} exists for $t\ge 0$ for every $\v\in \BR$, and satisfies that $\v(t)\in \BR$ for $t\ge 0$.


The equation \eqref{diffeqgeneralform} was shown in \cite{AlonsoStuartLongtime} and \cite{Dataassimilation} to be applicable to three chaotic dynamical systems, the Lorenz 63' model, the Lorenz 96' model, and the Navier--Stokes equations on the torus; such models have many applications. 
We note that instead of the trapping ball assumption, they consider different assumptions on $\mtx{A}$ and $\mtx{B}(\v,\v)$. As we shall explain in Section \ref{SecPreliminaries}, their assumptions imply \eqref{eqtrappingball}, and thus the trapping ball assumption is more general.


This article will consider results associated to the concentration properties
of the smoother and filter. In particular, 
for the geometric model of the Lorenz equations, as well as some other chaotic dynamical systems the following is established. In case of uniform observation noise, the diameter of the smoother and the filter are random variables depending on the observations. We show that their expected value remains lower bounded by a constant times the standard deviation of the noise, independently  of the number of observations. In the case of Gaussian observation noise, we show similar results for the diameter of the region of points whose likelihood is no smaller than a constant times likelihood at the true position. In addition, for the geometric model, under uniform noise assumption, we show that asymptotically in time, the smoother concentrates around a small line segment whose length is proportional to the standard deviation of the noise. Due to the substantial complexity of the dynamics of chaotic systems, such as the Lorenz 63' model, even the simple property of the sensitivity to the initial conditions have been only recently established by Tucker in \cite{Tucker2002}.
In this work a complex computer assisted proof was developed. We have only rigorously verified our assumptions required for the lower bounds for the geometric model of the Lorenz 63' equations. However, in order to show the practical relevance of our work, we include some numerical illustrations of the assumptions that are adopted, that seem to justify them in case of the Lorenz 63' and 96' models. It is stressed that establishing the conditions in such scenarios seems to require a concerted effort, which is beyond the scope of the current work.

We also consider upper bounds. For bounded noise distributions, under rather general conditions on the dynamics, the observation operator and the number of observations, the diameter of the support of the smoothing and filtering distributions are upper bounded by a constant times the standard deviation of the noise. This is generalised to noise distributions with unbounded support, where it is shown that the mean square error of some appropriate estimators for the initial position are of the same order as the variance of the observation noise. The assumptions required by these results are rigorously checked for the Lorenz 63' and Lorenz 96' models. We also check them for the case of randomly chosen coefficients. 

The lower bounds essentially tell us what is the best possible theoretical precision achievable by filtering/smoothing methods. They suggest that for such deterministic chaotic dynamical systems, noisy observations that are far in the future (or far in the past) typically do not contain much information that is useable for more accurate estimation of the initial position (or the current position, respectively).
These novel results are, to the best of our knowledge, the first in this area. They are also perhaps quite surprising, given the structure of the dynamical system.
The upper bounds imply that high precision filtering and smoothing is theoretically possible in almost every partially observed deterministic dynamical system given sufficiently precise observations (the only requirement is that given sufficient amount of noise-free observations, the initial position of the system is uniquely determined).

The structure of the paper is as follows. In Section \ref{SecPreliminaries}, we show some preliminary results about dynamical systems of the form \eqref{diffeqgeneralform}.
Section \ref{secgeometric} introduces the Lorenz 63' equations, and their corresponding geometric model. 
Our lower bounds for the geometric model are also presented in this section.
Section \ref{seclowerbounds} generalises the results to a larger class of dynamical systems. We state results for both uniform and Gaussian additive observation errors.  In Section \ref{secupperbounds}, we give upper bounds for the smoothing and filtering distributions for partially observed dynamical systems of the form \eqref{diffeqgeneralform}. The Appendix contains the proofs of a few technical lemmas for the geometric model and the proofs of some lower bounds based on assumptions on the return map to a plane.

\subsection{Preliminaries}\label{SecPreliminaries}
We now give some notations and basic properties of systems of the form \eqref{diffeqgeneralform} for use in the later sections. 

The one parameter solution semigroup will be denoted by
$\Psi_t$, thus for a starting point $\u\in \R^d$, the solution of \eqref{diffeqgeneralform} will be denoted by 
$\Psi_t(\u)$, or equivalently, $\u(t)$. The coordinates of the solution of \eqref{diffeqgeneralform} will be denoted by $\Psi_t^1(\u),\ldots, \Psi_t^d(\u)$, or equivalently,  $u_1(t),\ldots, u_d(t)$.

\cite{AlonsoStuartLongtime} and \cite{Dataassimilation} have assumed that the nonlinearity is energy conserving, i.e. $\left<\mtx{B}(\v,\v),\v\right>=0$ for every $\v\in \R^d$. They also assume that the linear operator $\mtx{A}$ is positive definite, i.e. there is a $\lambda_{\mtx{A}}>0$ such that $\left<\mtx{A}\v,\v\right>\ge \lambda_{\mtx{A}} \left<\v,\v\right>$ for every $\v\in \R^d$. As explained on page 50 of \cite{Dataassimilation}, \eqref{diffeqgeneralform} together with these assumptions above implies that for every $\v\in \R^d$,
\begin{equation}\label{eqabsorbingset}
\frac{1}{2}\frac{d}{dt} \|\v\|^2 \le \frac{1}{2\lambda_{\mtx{A}}}\|\vct{f}\|^2-\frac{\lambda_{\mtx{A}}}{2}\|\v\|^2.
\end{equation}
From \eqref{eqabsorbingset} one can show that set $\mathcal{B}_R$ is an absorbing set for any $R>\frac{\|\vct{f}\|}{\lambda_{\mtx{A}}}$ (thus all paths enter into this set, and they cannot escape from it once they have reached it). This in turn implies the existence of a global attractor (see e.g. \cite{Temam}, or Chapter 2 of \cite{Stuartdynamicalsystems}). Moreover, the trapping ball assumption \eqref{eqtrappingball} holds. In the two applications considered in this paper (the Lorenz 63' and 96' models), the energy conserving property of $\mtx{B}$ and the positive definiteness of $\mtx{A}$ were checked in \cite{LawAnalysis3DVAR63} and \cite{law2014controlling}, respectively.

For a differentiable function $g: \R^{d} \to \R^{d}$ with components $g(\v)=(g_1(\v),\ldots,g_{d}(\v))$, we define its Jacobian for every $\v\in \R^d$, denoted by $\mtx{J} g (\v)$ or equivalently $\J_{\v}(g)$, as a $d\times d$ matrix with elements $\left(\frac{\partial g_i(\v)}{\partial v_j}\right)_{1\le i,j\le d}$ (so the $i$th row contains the partial derivatives of $g_i$).

Based on \eqref{diffeqgeneralform}, we have that for any two points $\v, \w\in \BR$, any $t\ge 0$,
\[\frac{d}{dt}(\v(t)-\w(t))=-\mtx{A} (\v(t)-\w(t)) -(\mtx{B}(\v(t),\v(t)-\w(t))-\mtx{B}(\w(t)-\v(t),\w(t))),\]
and therefore by Gr\"{o}nwall's inequality, we have that for any $t\ge 0$,
\begin{equation}\label{eqpathdistancebound}\exp( -G t)\|\v-\w\|\le \|\v(t)-\w(t)\|\le \exp( G t)\|\v-\w\|,\end{equation}
for a constant 
\begin{equation}\label{Gdef}
G:=\|\mtx{A}\|+2\|\mtx{B}\| R,
\end{equation}
where $\|\mtx{A}\|$ denotes the $L^2$ norm of $\mtx{A}$, and $\|\mtx{B}\|:=\sup_{\v,\w\in \R^d: \|\v\|=1, \|\w\|=1} \|\mtx{B}(\v,\w)\|$.

Let $\Psi_{t_k}(\BR):= \{\Psi_{t_k}(\v): \v\in \BR\}$, then from inequality \eqref{eqpathdistancebound}, it follows that $\Psi_{t_k}:\BR\to \Psi_{t_k}(\BR)$ is a one-to-one mapping, which has an inverse that we are going to denote as $\Psi_{-t_k}: \Psi_{t_k}(\BR)\to \BR$.

We are going to describe next our assumptions about the observations. The system is observed at time points $t_j=j h$ for $j=0,1,\ldots$, with observations $\vct{Y}_j:=\mtx{H} \u(t_j) + \vct{Z}_j$, where $\mtx{H}: \R^d \to \R^{d_o}$ is a linear operator, and $(\vct{Z}_j)_{j\ge 0}$ are i.i.d.~centered random vectors taking values in $\R^{d_o}$ describing the noise. We assume that these vectors have distribution $\eta$ that is absolutely continuous with respect to the Lebesgue measure. We assume a prior $q$ on the initial condition, that is absolutely continuous with respect to the Lebesgue measure, and zero outside the ball $\BR$ (where the value of $R$ is determined by the trapping ball assumption \eqref{eqtrappingball}).

The main quantities of interest of this paper are the smoothing and filtering distributions corresponding to the conditional distribution of $\u(t_0)$ and $\u(t_k)$, respectively, given the observations $\vct{Y}_0,\ldots,\vct{Y}_k$. 
The densities of these distributions will be denoted by $\musm(\v|\vct{Y}_0,\ldots,\vct{Y}_k)$ and $\mufi(\v|\vct{Y}_0,\ldots,\vct{Y}_k)$, and they can be expressed as
\begin{align}
\label{eqmusm}&\musm(\v|\vct{Y}_0,\ldots,\vct{Y}_k)=\left[\prod_{i=0}^k \eta\left(\vct{Y}_i -\mtx{H}\Psi_{t_i}(\v)\right)\right]\cdot q(\v)/Z_{k}^{\mathrm{sm}}\text{ for }\v\in \BR, \text{and }\\
\nonumber&\musm(\v|\vct{Y}_0,\ldots,\vct{Y}_k)=0\text{ for }\v\notin \BR\\
\label{eqmufi}&\mufi(\v|\vct{Y}_0,\ldots,\vct{Y}_k)\\
\nonumber&=\left[\prod_{i=0}^k\eta\left(\vct{Y}_i-\mtx{H}\Psi_{t_i-t_k}(\v) \right)\right]\cdot \mathrm{det}(\mtx{J}\Psi_{-t_k}(\v)) \cdot q(\Psi_{-t_k}(\v))/Z_{k}^{\mathrm{fi}} \,\text{ for }\,\v\in \Psi_{t_k}(\BR), \text{ and}\\
\nonumber&\mufi(\v|\vct{Y}_0,\ldots,\vct{Y}_k)=0 \,\text{ for }\,\v\notin \Psi_{t_k}(\BR),
\end{align}
where $\mathrm{det}$ stands for determinant,  and $Z_{k}^{\mathrm{sm}}, Z_{k}^{\mathrm{fi}}$ are normalising constants independent of $\v$.  Since the determinant of the inverse of a matrix is the inverse of its determinant, we have the equivalent formulation
\begin{equation}\label{detalternativeq}
\mathrm{det}(\mtx{J}\Psi_{-t_k}(\v)) =\left(\mathrm{det}(\mtx{J}_{\Psi_{-t_k}(\v)}\Psi_{t_k})\right)^{-1}.
\end{equation}
In accordance with the usual definition in the literature, we will call the \emph{support of the smoother} the set of points $\v$ in $\BR$ where the density \eqref{eqmusm} is non-zero (and analogously where \eqref{eqmufi} is non-zero for the filter).

For $t\ge 0$, let $\v(t)$ denote the solution of \eqref{diffeqgeneralform} started from some $\v\in \R^d$.   Using  \eqref{diffeqgeneralform} and \eqref{eqtrappingball}, we have that 
\begin{align}\label{eqvmax}&\sup_{\v\in \BR, t\ge 0}\left\|\frac{d \v(t)}{d t}\right\|\le v_{\max}:=\|\mtx{A}\|R+\|\mtx{B}\|R^2+\|\vct{f}\|, \text{ and}\\
&\label{eqamax}\sup_{\v\in \BR}\left\|\mtx{J}_{\v} \left(\frac{d \v}{d t}\right)\right\|\le a_{\max}:=\|\mtx{A}\|+2\|\mtx{B}\|R.
\end{align}
By induction, we can show that for any $i\ge 2$, and any $\v\in \R^d$, we have
\begin{equation}\label{udereq}\D^i \v=-\mtx{A}\cdot \D^{i-1}\v-\sum_{j=0}^{i-1}  {i-1 \choose j} \mtx{B}\left(\D^j,\D^{i-1-j}\v\right).
\end{equation}
From this, it follows that for any $i\ge 0$, $\v\in \BR$ we have
\begin{align}\label{uderboundeq}
\left\|\D^i \v\right\|&\le C_0 \Cder^i \cdot i!, \text{ and }\\
\label{ugradboundeq}\left\|\mtx{J}_{\v} \left(\D^i \v\right)\right\|&\le C_{\mtx{J}}^i \cdot i!,
\end{align}
where $C_0:=R+\frac{\|\vct{f}\|}{\|\mtx{A}\|}$, $\Cder:=\|\mtx{A}\|+\|\mtx{B}\|R+\frac{\|\mtx{B}\|}{\|\mtx{A}\|} \|\vct{f}\|$, and
$C_{\mtx{J}}:=2\Cder$. To see this, it suffices to first verify \eqref{uderboundeq} and \eqref{ugradboundeq} for $i=0$ and $i=1$, and then use induction and the recursion formula \eqref{udereq} for $i\ge 2$. It is possible to prove the existence and finiteness of $\J\Psi_{t_k}(\v)$ for any $\v\in \BR$, $t_k\ge 0$ based on \eqref{ugradboundeq} and the Taylor expansion (if $t_k<C_{\mtx{J}}^{-1}$, then the Taylor expansion converges, while if $t_k\ge C_{\mtx{J}}^{-1}$, then we can write it as $t_k=a_1+\ldots+a_l$ for some $0< a_1,\ldots, a_l<C_{\mtx{J}}^{-1}$, and use the chain rule in computing $\J\Psi_{a_1}(\ldots \Psi_{a_l}(\v)\ldots)$ ).

\subsection{Comparison with the work of Lalley and Nobel}\label{seccomparisonlalley}
\cite{Lalleybeneaththenoise} has studied the statistical behaviour of hyperbolic maps, and the results were extended in \cite{LalleyNobel} under weaker assumptions. They consider invertible maps $F:\Lambda\to \Lambda$ for some compact set $\Lambda\in \R^d$, such that $F^{(i)}(x)$ exists for every $i\in \Z$ and $x\in \Lambda$ ($F^{(i)}(x)=(F^{-1})^{(i)}(x)$ for $i<0$). In practice this means that $\Lambda$ is usually chosen as an attractor of the system. 

\cite{LalleyNobel} calls an invertible map $F: \Lambda\to \Lambda$ \emph{expansive} if there is an absolute constant $\Delta>0$ such that for every $\v,\v'\in \Lambda$ with $\v\ne \v'$, $\sup_{s\in \Z} \|F^{(s)}(\v)-F^{(s)}(\v')\|>\Delta$. 
Based on this assumption (which can be proven for hyperbolic maps), they show that if we observe the system with bounded observation noise whose maximum size $\epsilon$ satisfies that $\epsilon\le \Delta/5$, then as we get more and more two sided observations (both from the future and the past), the position of the system can be determined with arbitrary precision. They propose an algorithm called Smoothing Algorithm D that allows one to recover the positions $F^{(k)}(\u),F^{(k+1)}(\u),\ldots, F^{(n-k)}(\u)$ given observations $\vct{Y}_{0:n}$ when $k$ and $n$ tends to infinity at the right rate. \cite{LalleyNobel} also considers Gaussian observation errors, and shows that under some conditions (in particular, for hyperbolic systems), even if we would have all the observations $(\vct{Y}_s)_{s\in \Z}$, there would still not exist any measurable function that recovers the initial position.

The results of the present paper differ from these earlier results in several ways. Firstly, we do not assume that the state space $\Lambda$ is invariant with respect to the map $F$, thus the inverse $F^{-1}$ and its iterates might not be defined at every point $\v\in \Lambda$ (indeed, the differential equation \eqref{diffeqgeneralform} cannot in general be solved backwards in time globally for every $\v\in \BR$). Secondly, we do not assume hyperbolicity, or expansiveness of the map. Finally, because we consider the smoothing and filtering distributions, i.e. the distributions of $\u(t_0)$ and $\u(t_n)$ given $\vct{Y}_{0:n}$, we do not have access to two sided observations, thus the Smoothing Algorithm D of \cite{LalleyNobel} and its variants  are not applicable. Therefore even if at first sight it might seem that our lower bounds contradict the fact that  \cite{LalleyNobel} proves that the path can be recovered if the size of the bounded noise if sufficiently small, this is due to the fact that they use two sided observations, while we do not. We have verified that even for Smale's solenoid mapping (an example of \cite{Lalleybeneaththenoise}), the lower bounds of Theorems \ref{thmlowerboundsmoother} and \ref{thmlowerbndfilter} are applicable.

Since our main object of interests are chaotic differential equations of the form \eqref{diffeqgeneralform}, our results are presented in terms of continuous time mappings, in contrast with the discrete time mappings of \cite{LalleyNobel}. Although they could be rewritten as discrete time mappings, we feel that this would introduce additional abstraction, and make the presentation less clear.

\section{The Lorenz equations, and their geometric model}\label{secgeometric}
In this section we study the behaviour of the smoothing and filtering distributions for the geometric model associated to the Lorenz equations. We introduce the model in Section \ref{secLorenzIntro}. This is followed by lower bounds on the diameter of the support of the smoother and filter, assuming bounded observation noise, deduced in Sections \ref{seclowerboundsgeometricsmoother} and \ref{seclowerboundsgeometricfilter}. Finally, we analyse the limit of the support of the smoothing distribution as the number of observations tends to infinity in Section \ref{secgeometricupperbounds}.

\subsection{Introduction to the model}\label{secLorenzIntro}
Lorenz has introduced the following system of equations in \cite{lorenz1963deterministic}, 
\begin{align}
\frac{d u_1}{dt}&=a(u_2-u_1),\label{lorenz63eq1}\\
\frac{d u_2}{dt}&=ru_1-u_2-u_1u_3,\label{lorenz63eq2}\\
\frac{d u_3}{dt}&=-bu_3+u_1u_2.\label{lorenz63eq3}
\end{align}
Lorenz has set the values of the parameters as $a=27$, $b=\frac{8}{3}$, and $r=10$. For these choice of parameters, it was observed that these equations have bounded solutions, but surprisingly, they are very sensitive to the choice of initial conditions. For almost every two starting points $\u$ and $\v$, the solutions $\u(t)$ and $\v(t)$ are eventually further apart than some absolute constant $\xi>0$ for some $t>0$. This chaotic behaviour was quite different from the behaviour of previously studied dynamical systems. Since then, considerable effort has been spent on understanding such systems, in particular due to the application of such models to weather forecasting. Rigorously justifying the chaotic behaviour for the original Lorenz equations nevertheless has proven to be a challenging problem, which was only settled rather recently by Tucker, who has given a computer assisted proof \cite{Tucker2002}. One key difficulty is the fact that the equations cannot be solved analytically. Another is that the solution might spend arbitrarily long time near the origin (which is a stationary point).

Since the chaotic behaviour of the Lorenz equations was difficult to analyse directly, \cite{guckenheimer1976strange} and \cite{afraimovich1977origin} have independently proposed the so-called \emph{geometric model} associated to the Lorenz equations. This is still a 3 dimensional dynamical system which can be described by time independent differential equations, and it was conjectured that it shares many features of the original equations. Due to its particular form, it is analytically solvable, and in \cite{guckenheimer1976strange} it was shown that it has sensitive dependence to initial conditions.

In this section we define the geometric model and describe some of its properties. The description is based on \cite{guckenheimer1976strange} and \cite{Lorenzlikeflows}. Although this is a rather simple analytically solvable model, we believe that its behaviour is similar to many other more complex chaotic systems (and, as we shall see in Section \ref{seclowerbounds}, we generalise some of the results obtained for this model to some other chaotic dynamical systems).

The geometric model of the Lorenz equations consists of two parts. In the first part, the flow is going downwards from a square $S$ to one of two cusps $\Sigma^+$ or $\Sigma^-$ (see Figure \ref{geometric:sub1}). In the second part, the flow is going upwards from these two cusps back to the square $S$ (see Figure \ref{geometric:sub2}).
Note that this flow is only defined for points inside a bounded set (consisting of the union of paths started from $S$ until they first return to $S$). In the following few paragraphs, we give a precise definition of the flow and explain how is it related to the Lorenz 63' equations.

One particular feature of the Lorenz equations is that near the origin, through conjugation they can be shown to be equivalent  to a linear system of the form
\[\left(\frac{d u_1}{dt},\frac{d u_2}{dt},\frac{d u_3}{dt}\right)=(\lambda_1 u_1, -\lambda_2 u_2, -\lambda_3 u_3), \text{ with }0<\lambda_3<\lambda_1<\lambda_2.\]
The solution of these equations is given by 
\begin{equation}\label{eqflow}
\Psi^{\mathrm{lin}}_t(\u)= \left(u_1 e^{\lambda_1 t}, u_2 e^{-\lambda_2 t}, u_3 e^{-\lambda_3 t}\right).
\end{equation}
This particular form means that nearby points can take arbitrarily long time to escape from the neighbourhood of the origin.

Let us denote the so called \emph{return square} by 
\[S:=\{(u_1,u_2,1): |u_1|\le 1/2, |u_2|\le 1/2\}.\] This square is in transverse direction to flow \eqref{eqflow}, which is going downwards in direction $u_3$ when passing through it. Let 
\begin{align*}
&S^-:=\{(u_1,u_2,1)\in S: u_1<0\}, S^+:=\{(u_1,u_2,1)\in S: u_1>0\},S^*=S^- \cup S^+,\\
&\Gamma:=\{(u_1,u_2,1)\in S: u_1=0\}, \Sigma:=\{(u_1,u_2,u_3): |u_1|=1\}.
\end{align*}
In the geometric model, the points started from $S$ start according to equations \eqref{eqflow} until they reach $\Sigma$ (the points on $\Gamma$ will converge to the origin and never reach $\Sigma$). 

Based on \eqref{eqflow}, we can see that the time it takes for a path started from a point $\u\in S^*$ to reach $\Sigma$ is $\tau_{\Sigma}(\u):=\frac{1}{\lambda_1} \log(1/|u_1|)$. The location of the exit point will be
\begin{align*}
\Psi^{\mathrm{lin}}_{\tau_{\Sigma}(\u)}(\u)&=\left(\sgn(u_1),u_2 e^{\lambda_2 \tau_{\Sigma}(\u)}, e^{\lambda_3 \tau_{\Sigma}(\u)} \right)\\
&=\left(\sgn(u_1), u_2  |u_1|^{\frac{\lambda_2}{\lambda_1}}, |u_1|^{\frac{\lambda_3}{\lambda_1}} \right).
\end{align*}
Let $\alpha:=\frac{\lambda_3}{\lambda_1}$ and $\beta:=\frac{\lambda_2}{\lambda_1}$, then $0<\alpha<1<\beta$, and 
\begin{equation}\label{Ldefeq}
L(\u):=\Psi^{\mathrm{lin}}_{\tau_{\Sigma}(\u)}(\u)=(\sgn(u_1),u_2 |u_1|^{\beta}, |u_1|^{\alpha}).
\end{equation}
As we can see on Figure \ref{geometric:sub1}, the function $L$ maps the two half squares $S^-$ and $S^+$ into cusps (triangles with curved edges). We denote these cusps by $\Sigma^-$ and $\Sigma^+$, respectively.

\begin{figure}
\begin{subfigure}{0.385\linewidth}
\centering
\includegraphics[width=\linewidth]{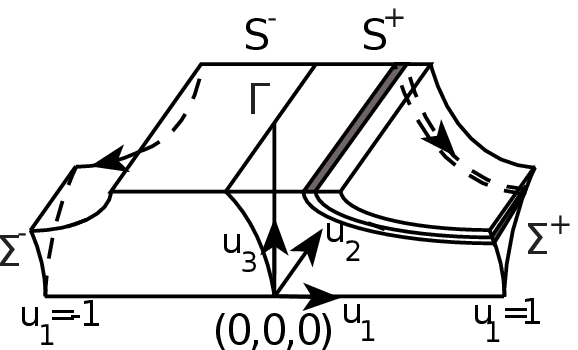}
\caption{Flow from $S^*$ to $\Sigma$}
\label{geometric:sub1}
\end{subfigure}
\begin{subfigure}{0.385\linewidth}
\centering
\includegraphics[width=\linewidth]{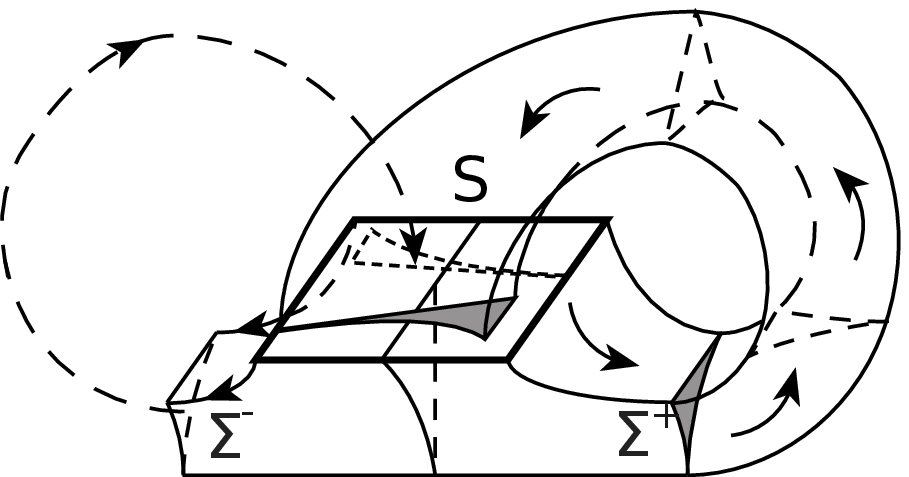}
\caption{Flow from $\Sigma$ to $S$}
\label{geometric:sub2}
\end{subfigure}
\begin{subfigure}{0.21\linewidth}
\centering
\includegraphics[width=\linewidth]{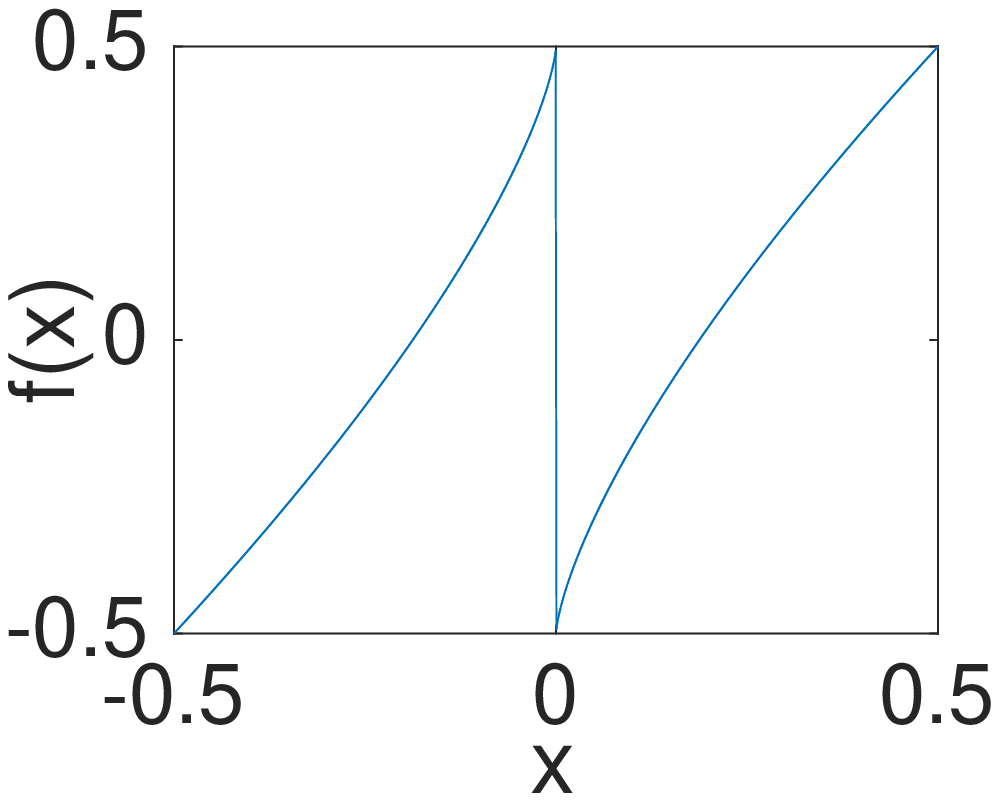}
\caption{The function $f$}
\label{geometric:sub3}
\end{subfigure}
\caption{Illustration of the geometric model of the Lorenz 63' equations}
\label{fig:test}
\end{figure}

The vertices of these cusps are given by
\begin{align*}
\Sigma^+_1&:=(1,0,0), \quad \Sigma^-_1:=(-1,0,0),\\
\Sigma^+_2&:=\left(1,\left(\frac{1}{2}\right)^{1+\beta},\left(\frac{1}{2}\right)^{\alpha}\right), \quad
 \Sigma^-_2:=\left(-1,\left(\frac{1}{2}\right)^{1+\beta},\left(\frac{1}{2}\right)^{\alpha}\right),\\
\Sigma^+_3&:=\left(1,-\left(\frac{1}{2}\right)^{1+\beta},\left(\frac{1}{2}\right)^{\alpha}\right), \quad
\Sigma^-_3:=\left(-1,-\left(\frac{1}{2}\right)^{1+\beta},\left(\frac{1}{2}\right)^{\alpha}\right).
\end{align*}

Once the paths have reached cusp $\Sigma^+$ (or $\Sigma^-$), they move back to the return square $S$ via a linear transformation which is a composition of a rotation around the line $(1,u_2,1)$ (or $(-1,u_2,1)$) by $\frac{3\pi}{2}$, an expansion in the $u_1$ direction by a factor $\theta$, and translation by $-\frac{1}{2}$ in the $u_1$ direction and by $-\frac{1}{4}$ in the $u_2$ direction (or by $\frac{1}{2}$ in the $u_1$ direction and by $\frac{1}{4}$ in the $u_2$ direction, respectively, for $\Sigma^-$). 

This means that a point $\v\in \Sigma$ will be mapped to the point on $S$ defined as
\begin{equation}\label{eqSigmatoS}
\Psi_{\Sigma\to S}(\v):=\bigg \{\begin{array}{ll} (\theta v_3-\frac{1}{2}, v_2-\frac{1}{4},1) &\text{ for }\v\in \Sigma^+,\\
 (-\theta v_3+\frac{1}{2}, v_2+\frac{1}{4},1)&\text{ for }\v\in \Sigma^-.\end{array}
\end{equation}
In order for the construction to be consistent (that is, none of the paths started at different points of $S$ intersect until their first return), we make the following assumptions on the eigenvalues $\lambda_1, \lambda_2, \lambda_3$ and the parameter $\theta$.

\begin{ass}\label{assalphabetatheta}
Suppose that
\begin{enumerate}
\item the coefficient $\alpha=\frac{\lambda_3}{\lambda_1}$ satisfies that $\frac{1}{\sqrt{2}}<\alpha<1$,
\item the coefficient $\beta=\frac{\lambda_2}{\lambda_1}$ satisfies that $\beta>\frac{\log(6)}{\log(2)}-1$,
\item the coefficient $\theta$ satisfies that $\frac{2^{\alpha}}{\sqrt{2}\alpha}<\theta<2^{\alpha}$.
\end{enumerate}
\end{ass}
This process is illustrated on Figure \ref{geometric:sub2}.

\cite{Lorenzlikeflows} has defined the three transformations (rotation, expansion, and translation) precisely. These specify $\Psi_{\Sigma\to S}(\v)$, however, the exact time evolution of the process from $\Sigma$ to $S$ was not given because this was not needed for the purpose of showing the sensitivity of the model with respect to initial conditions (except that they have assumed that we reach $S$ from any point on $\Sigma$ in a  bounded amount of time). For the sake of completeness, here we make a specific choice of this evolution. Any point on $\Sigma^+$ (or $\Sigma^-$) will take $\frac{3\pi}{2}$ time to reach the return square $S$ (the time parameter $t$ expresses the angle of the rotation). For the evolution of the points of $\Sigma^+$, we will use the polar coordinate system
\[(u_1,u_2,u_3)=(1+r\sin(\phi), u_2, 1-r\cos (\phi)).\]
In this coordinate system, $\phi$ represents the angle of rotation we have done along the line $(1,u_2,1)$, $r$ represents the distance from the line $(1,u_2,1)$, and finally $u_2$ represents the $u_2$-coordinate.

The evolution of the angle $\phi$ can be chosen linearly in time, that is, $\phi(s)=s$ for $0\le s\le (3/2)\pi$. The transformation of the coordinate $u_2$ from $u_2$ to $u_2-\frac{1}{4}$ can be defined to happen linearly in time too, that is, 
\begin{equation}\label{yeveq}
u_2(s)=u_2-\frac{1}{4}\cdot \frac{s}{(3/2)\pi}.
\end{equation}
Finally, due to \eqref{eqSigmatoS}, the evolution of $r(s)$ has to satisfy the conditions that
$r(0)=1-u_3$ and $r\left(\frac{3}{2}\pi\right)=\frac{3}{2}-\theta u_3$. These are satisfied by the linear interpolation 
\begin{equation}\label{reveq}
r(s):=1-u_3+\left(\frac{1}{2}-(\theta-1) u_3\right)\cdot \frac{s}{(3/2)\pi} \text{ for }0\le s\le \frac{3}{2}\pi.
\end{equation}
Thus the flow from $\Sigma^+$ to $S$ for time $0\le s\le (3/2)\pi$ is given by the equations
\begin{equation}\label{SigmaptoSeveq}
\Psi_s^{\mathrm{rot}}(\u):=\left(1+r(s) \sin(s), u_2-\frac{1}{4}\cdot \frac{s}{(3/2)\pi},  1-r(s)\cos (s)\right) \text{ for }\u\in \Sigma^+.
\end{equation}
Similarly, using the same definition of $r(s)$, we can write the flow from $\Sigma^-$ to $S$ for time $0\le s\le (3/2)\pi$ as
\begin{equation}\label{SigmamtoSeveq}
\Psi_s^{\mathrm{rot}}(\u):=\left(-1-r(s) \sin(s), u_2+\frac{1}{4}\cdot \frac{s}{(3/2)\pi},  1-r(s)\cos (s)\right)\text{ for }\u\in \Sigma^-.
\end{equation}
It is not difficult to see that these two flows do not intersect at any time point $0\le s\le (3/2)\pi$. Firstly, for $0\le s\le \pi$, we have $u_1(s)\ge 1$ for the flow started from $\Sigma^+$, and $u_1(s)\le -1$ for the flow started from $\Sigma^-$. For the flow started at $\Sigma^+$, we have $u_2\le \left(\frac{1}{2}\right)^{1+\beta}$, so for $\pi<s\le (3/2)\pi$, we have $u_2(s)<\left(\frac{1}{2}\right)^{1+\beta}-\frac{1}{4}\cdot \frac{2}{3}<0$   by Assumption \ref{assalphabetatheta}. It can be shown similarly that $u_2(s)>0$ for $\pi<s\le (3/2)\pi$ for the flow started from $\Sigma^-$. Therefore the two flows started at $S^+$ and $S^-$, respectively cannot intersect until their return to $S$.

By the definition of the model, the return times from $\u\in S^*$ to $S$ are given by 
\begin{equation}\label{eqreturntime}
\tau(\u):=\tau_{\Sigma}(\u)+\frac{3}{2}\pi = \frac{1}{\lambda_1} \log(1/|u_1|) + \frac{3}{2}\pi. 
\end{equation}
The \emph{semigroup} of the dynamics of the geometric model, $\Psi^{\mathrm{geo}}_t(\u)$, consists of repeated compositions of the semigroup $\Psi^{\mathrm{lin}}$ from $S^*$ to $\Sigma$ and $\Psi^{\mathrm{rot}}$ from $\Sigma$ back to $S$. 

The \emph{state space} where $\Psi^{\mathrm{geo}}_t(\u)$ is defined is denoted by $\Lambdageo$, which consists of the union of the points of all of the paths started from $S$ and evolved according to the geometric model until their first return to $S$ (the paths started from points on $\Gamma$ do not return to $S$, but the points on them are included in $\Lambdageo$ nevertheless).

The dynamics $\Psi^{\mathrm{geo}}_t(\u)$ defines a return map $P(\u)$ from $S^*$ to $S$. An important property of the return map $P(\u)$ is that two points that were equal in $u_1$ coordinate stay equal in $u_1$ coordinate even after their return. Thus the $u_1$ coordinate of $P(\u)$ only depends on $u_1$, and thus we can write
\begin{equation}\label{Peq}
P(\u):=(f(u_1),g(u_1,u_2)),
\end{equation}
where $f: [-1/2,1/2]\setminus \{0\}\to [-1/2,1/2]$  is defined as
\begin{equation}\label{feq}
f(u_1):=\Bigg\{\begin{matrix}\theta |u_1|^{\alpha}-\frac{1}{2} &\text{ if }u_1>0,\\
-\theta |u_1|^{\alpha}+\frac{1}{2}& \text{ if }u_1<0,
\end{matrix}
\end{equation}
and $g: ([-1/2,1/2]\setminus \{0\})\times [-1/2,1/2]\to [-1/2,1/2]$ is defined as
\begin{equation}\label{geq}
g(u_1,u_2):=\Bigg\{\begin{matrix} u_2 |u_1|^{\beta}-\frac{1}{4} &\text{ if }u_1>0,\\
u_2 |u_1|^{\beta}+\frac{1}{4} &\text{ if }u_1<0.
\end{matrix}
\end{equation}

Figure \ref{geometric:sub3} displays $f$. Based on Assumption \ref{assalphabetatheta}, one can see that this function satisfies $|f'(u_1)|>\sqrt{2}$ on $[-1/2,1/2]\setminus \{0\}$. This means that the dynamics are expanding in the direction $u_1$ and this causes the high sensitivity to initial conditions. The following result summarises some important statistical properties of the map $f$.
\begin{prop}[Proposition 2.2 of \cite{Lorenzlikeflows}]\label{Prop22}
The one-dimensional map $f$ admits a unique invariant probability distribution $\mu_f$ on $[-1/2,1/2]$ that is absolutely continuous with respect to the Lebesgue measure $m$ on the interval, it is ergodic and so in particular it is a physical measure for the map. Moreover, $\frac{d \mu_f}{d m}$ is of bounded variation, in particular, it is bounded.
\end{prop}

Now we are going to state two more useful properties of the geometric model.
Firstly, based on equations \eqref{eqflow} and \eqref{SigmaptoSeveq}, one can show that the speed of the dynamics $\|\frac{d\u}{dt}\|$ at any $\u\in \Lambdageo$ is bounded by
\begin{equation}\label{vmaxgeodefeq}\vmaxgeo:=4+\lambda_1+\lambda_2+\lambda_3.
\end{equation}
For $s\ge 0$, we let
\begin{equation}\label{WSsdefeq}
W^{S}_{s}:=\{\w\in \Lambdageo: w_1 \in [-1,1], w_2 \in [-1,1], w_3\in [1-s, 1+s]\},
\end{equation}
this is the region of points in $\Lambdageo$ not further away than $s$ from the plane $S$. Based on equations \eqref{eqflow} and \eqref{SigmaptoSeveq}, it is possible to show that for any $\u\in W^{S}_{0.1}\cap \Lambdageo$, the dynamics of the geometric model satisfies that 
\begin{equation}\label{vmingeodefeq}
\frac{d u_3}{d t} \le -\vmingeo\quad \text{ for } \quad \vmingeo:=\min\left(\frac{1}{4}, 0.9 \lambda_3 \right).
\end{equation}

\subsection{Lower bounds for the smoother of the geometric model}\label{seclowerboundsgeometricsmoother}
In this section, we give some lower bounds for the smoothing distribution of the geometric model of the Lorenz equations. First, we show the existence of the so-called \emph{leaf sets}, a rather surprising property of the dynamics of the geometric model.
\begin{theorem}\label{thmlowerboundsmoother1}
For Lebesgue almost every point $\u\in \Lambdageo$, there exists a continuous curve $U(\u)\subset \Lambdageo$ called a \emph{leaf set} such that $\{\|\v-\u\|_{\infty}: \v \in U(\u)\}=[0,d_{\max}(\u)]$ for some constant $d_{\max}(\u)>0$, and for any $\v\in U(\u)$, any $t>0$, 
\begin{equation}\label{eqpathdistance}\|\v(t)-\u(t)\|_{\infty}\le C_g \|\v-\u\|_{\infty} \exp(-\lambda_g \cdot t),
\end{equation}
where $C_g>0$ and $\lambda_g>0$ are constants only depending on the parameters of the model. Moreover, we can choose 
\begin{equation}\label{eqUudef}U(\u):=\left\{\v\in \Lambdageo: v_1=u_1, v_3=u_3, |v_2-u_2|<\frac{1}{3}-\frac{1}{2^\beta}\right\}.
\end{equation}
\end{theorem}

Thus the leaf set $U(\u)$ satisfies that for any $\v\in U(\u)$, the distance between the paths $\u(t)$ and $\v(t)$ decreases rapidly in $t$. This is a rather unusual property since in general two paths started from nearby points diverge quickly. Using this, we obtain our lower bound for the smoother.

\begin{theorem}\label{thmlowerboundsmoother2}
Suppose that we observe the geometric model started at $\u$ at time points $t_i=ih$ for $i=0,1,\ldots$ with observation matrix $\mtx{H}$, and that the observation errors are uniformly distributed on $[-\epsilon, \epsilon]^{d_o}$. Suppose that the prior $q$ satisfies that $q(\v)>0$ for every $\v\in \Lambdageo$. Then for Lebesgue almost every initial point $\u\in \Lambda$, for $\epsilon\cdot h$ sufficiently small, the smoothing distribution given the observations up to time $t_k$ for any $k\in \N$ satisfies that the expected diameter of its support is at least $\frac{c_{\mathrm{sm}} h \epsilon}{\|\mtx{H}\|_{\infty}}$ for some constant $c_{\mathrm{sm}}>0$ only depending on the parameters of the model.
\end{theorem}
\begin{rem}
To make the argument transparent, we only consider the uniform case here, but the result could be easily generalised to other observation error distributions with bounded support. \end{rem}

\begin{proof}[Proof of Theorem \ref{thmlowerboundsmoother1}]
Let $U(\u)$ be as in \eqref{eqUudef}. Then $\v\in U(\u)$ and $\u$ can only differ in the second coordinate. Using the condition that $|v_2-u_2|<\frac{1}{3}-\frac{1}{2^{\beta}}$, it follows that it cannot happen that $\v$ is an element of the flow from $\Sigma_+$ to $S$ while $\u$ is an element of the flow from $\Sigma_-$ to $S$, or vice-versa (since the two flows are at least $\frac{1}{3}-\frac{1}{2^{\beta}}$ away in the second coordinate in the region above $S$). Using this fact, and the definition of the dynamics, we can see that the second coordinate does not influence the evolution of the first and third coordinates, thus $v_1(t)=u_1(t)$ and $v_3(t)=u_3(t)$ for every $t\ge 0$. Now from \eqref{eqflow}, it follows that the difference in the second coordinate decreases at a rate $e^{-\lambda_2 t}$ during the flow from $S$ to $\Sigma$. Moreover, the time it takes to get from $S^*$ to $\Sigma$ is at least $\frac{1}{\lambda_1}\log(2)$. After this period of contraction, we can see that the dynamics keeps $\|\u(t)-\v(t)\|_{\infty}$ constant during the phase from $\Sigma$ back to $S$, which takes $\frac{3}{2}\pi$ time. By combining these facts, the result follows with constants 
$\lambda_g:=\lambda_2\cdot \frac{\log(2)/\lambda_1}{\log(2)/\lambda_1+(3/2)\pi}$ and $C_g:=\exp\left(\frac{3}{2}\pi \lambda_g\right)$.
\end{proof}

\begin{proof}[Proof of Theorem \ref{thmlowerboundsmoother2}]
Suppose that we observe a point $\vct{x}\in \R^d$ with observation error that is uniform in $[-\epsilon,\epsilon]^{d_o}$, and obtain an observation $\vct{Y}$. Then for any $\vct{z}\in \R^{d_o}$ with $\|\vct{z}-\vct{x}\|_{\infty}< \epsilon$, we have
\begin{equation}\label{equniformobs}\PP(\|\vct{z}-\vct{Y}\|_{\infty}>\epsilon)\le \PP(\|\vct{x}-\vct{Y}\|_{\infty}>\epsilon-\|\vct{z}-\vct{x}\|_{\infty})\le \frac{d_o\|\vct{z}-\vct{x}\|_{\infty}}{2\epsilon}.
\end{equation}
Outside of this event, $\vct{z}$ is still within the support of the posterior distribution. Based on this observation, and inequality \eqref{eqpathdistance}, we can see that the probability that a point $\v\in U(\u)$ is not in the support of the smoothing distribution with observations taken into account until time $t$ is bounded by
\[\sum_{i=1}^{j}\|\mtx{H}\|_{\infty}C_g \frac{d_o\|\v-\u\|_{\infty} \exp(-\lambda_g i h)}{2\epsilon}\le \|\v-\u\|_{\infty}\cdot \frac{d_o C_g \|\mtx{H}\|_{\infty}}{2\lambda_g h \epsilon}.\]
Thus the probability that a point $\v\in U(\u)$ is in the support of the smoothing distribution given any amount of observations is at least $\frac{1}{2}$ if $\|\v-\u\|_{\infty}\le \frac{\epsilon \lambda_g h}{d_o \|\mtx{H}\|_{\infty} C_g}$. Let $m(\u):=\sup_{v\in U(\u)}\|\v-\u\|_{\infty}$. Then $m(\u)>0$ for Lebesgue almost every $\u\in \Lambda$, and assuming that $\epsilon h\le \frac{d_o \|\mtx{H}\|_{\infty} C_g m(\u)}{\lambda_g}$, there is a $\v^*\in U(\u)$ such that $\|\v^*-\u\|_{\infty}=\frac{\epsilon h \lambda_g }{d_o \|\mtx{H}\|_{\infty} C_g}$. Since $\u$ is in the support of the smoother, and $\v^*$ is included with probability at least $\frac{1}{2}$, therefore the expected diameter of the smoother is at least $\frac{c_{\mathrm{sm}}\epsilon h }{\|\mtx{H}\|_{\infty}}$, where $c_{\mathrm{sm}}:= \frac{ \lambda_g }{2 d_o C_g}$.
\end{proof}

\subsection{Lower bounds for the filter of the geometric model}\label{seclowerboundsgeometricfilter}

Our first theorem in this section shows the existence of the so-called \emph{anti-leaf sets}. For any $\v\in \Lambdageo$, we call
\begin{equation}\label{originofveq} O(\v)=(O_1(\v),O_2(\v),O_3(\v))\in S
\end{equation} the \emph{origin} of $\v$ on $S$. This is defined as the unique point in $S$ such that if we start the the geometric model (as defined in Section \ref{secLorenzIntro}) from $O(\v)$, its path will cross $\v$ before returning to $S$. The time taken to reach $\v$ from $O(\v)$ is denoted by $\tau_O(\v)$. For any $\u, \v\in \Lambdageo$, $t\ge 0$, we let 
$\u(t):=\Psi_{t}^{\mathrm{geo}}(\u)$ and $\v(t):=\Psi_{t}^{\mathrm{geo}}(\v)$  (it is evolved according to the geometric model for time $t$).
\begin{theorem}\label{thmgeoantileafsets}Then there is an absolute constant $h_{\max}>0$ such that for $h\le h_{\max}$, for $\mu_f$-almost every $x\in [-1/2,1/2]$ ($\mu_f$ was defined in Proposition \ref{Prop22}), for every $\u\in \Lambdageo$ with $O_1(\u)=x$, for every $k\in \N$, there exists a sequence of continuous curves $\tU(\u,k)\subset\Lambdageo$ (called \emph{anti-leaf sets}) and constants $\tilde{d}_{\max}(\u,k)\ge 0$, $C_{\tU}(\u,k)<\infty$ such that 
\begin{enumerate}
\item $\{\|\v(t_k)-\u(t_k)\|_{\infty}: \v\in \tU(\u,k)\}=[0, \tilde{d}_{\max}(\u,k)]$, 
\item $\sum_{i=0}^k \|\v(t_i)-\u(t_i)\|_{\infty}\le C_{\tU}(\u,k)\|\v(t_k)-\u(t_k)\|_{\infty}$, \text{ and }
\item there is an infinite sequence of indices $i_1(\u), i_2(\u), \ldots $ such that for any $j\ge 1$, $C_{\tU}(\u,i_j(\u))\le C_{\tU}(\u)$ and $\tilde{d}_{\max}(\u,i_j(\u))\ge \tilde{d}_{\max}(\u)$, where
$\tilde{d}_{\max}(\u)>0$ and $C_{\tU}(\u)<\infty$ are some constants that are independent of $j$.
\end{enumerate}
\end{theorem}

The anti-leaf sets behave the opposite way to the leaf set considered in the previous section, because for $\v\in \tilde{U}(\u,k)$, the distance $\|\v(t)-\u(t)\|_{\infty}$ increases rapidly in $t$ for $0\le t\le t_k$. This is the typical behaviour of paths of a chaotic system started from nearby points, so their existence is not surprising. Nevertheless, the proof of Theorem \ref{thmgeoantileafsets} is quite technical, so we have included it in Section \ref{secgeoantileafproof} of the Appendix. The key idea is that we can exploit the expansion property of the one dimensional map $f$ by looking at the time evolution of a small line segment parallel to the axis $u_1$ passing through $O(\u)$.

Based on the existence of anti-leaf sets, the following theorem shows lower bounds for the diameter of the filtering distribution for the geometric model.
\begin{theorem}\label{thmgeomodelfilter}
Suppose that we observe the geometric model started at position $\u\in \Lambdageo$, with observation matrix $\mtx{H}$, at time points $t_i=ih$ for $i=0,1,\ldots$, with observation errors that are uniform on $[-\epsilon, \epsilon]^{d_{o}}$, and $h\le h_{\max}$ (defined as in Theorem \ref{thmgeoantileafsets}). Suppose that the prior $q$ satisfies that $q(\v)>0$ for every $\v\in \Lambdageo$. Then for $\mu_f$-almost every $x\in [-1/2,1/2]$, every $\u\in \Lambdageo$ with $O_1(\u)=x$, for every $j\ge 1$, any $0<\epsilon \le d_oC_{\tU}(\u)\tilde{d}_{\max}(\u)\left\|\mtx{H}\right\|_{\infty}$, the expected diameter of the support of the filter after observations up to time $t_{i_j(\u)}$ is larger than or equal to $\epsilon/(2 d_o C_{\tU}(\u) \left\|\mtx{H}\right\|_{\infty})$.
\end{theorem}
Thus the theorem states that for infinitely observation times $t_{i_1(\u)},  t_{i_2(\u)},\ldots$, the expected diameter of the support of the filter is lower bounded by a constant times the standard deviation of the noise, and thus it does not tend to a Dirac-$\delta$ around the current position. Note that this result is weaker than our lower bound for the smoother (Theorem \ref{thmlowerboundsmoother2}) in the sense that it only holds at some specific time points and not for every $t_k$. Indeed, for the geometric model, the path $\u(t)$ can approach the origin $(0,0,0)$ infinitely often, and its speed $\|\frac{d}{dt}\u(t)\|$ can get arbitrarily slow in the neighbourhood of the origin. At such positions, the filtering distribution can get highly concentrated, since we have many independent observations about positions that are very close to the current position. Therefore one cannot expect a time uniform lower bound of the same form as for the smoother.
\begin{proof}[Proof of Theorem \ref{thmgeomodelfilter}]
Using the condition that $\epsilon \le  d_o C_{\tU}(\u)\tilde{d}_{\max}(\u)\left\|\mtx{H}\right\|_{\infty}$, based on Theorem \ref{thmgeoantileafsets}, for any $j\ge 1$, there is a point $\v\in \tU(\u,i_j(\u))$ satisfying that 
\[\|\v(t_{i_j(\u)})-\u(t_{i_j(\u)})\|_{\infty}=\frac{\epsilon}{d_o C_{\tU}(\u)\left\|\mtx{H}\right\|_{\infty}}.\]
For this $\v$, by Theorem \ref{thmgeoantileafsets}, we have
\[\sum_{i=0}^{i_j(\u)} \|\v(t_i)-\u(t_i)\|_{\infty}\le C_{\tU}(\u)\|\v(t_{i_j(\u)})-\u(t_{i_j(\u)})\|_{\infty}\le \frac{\epsilon}{d_o \left\|\mtx{H}\right\|_{\infty}}.\]
Using \eqref{equniformobs} and the union bound, the probability that $\v(t_{i_j(\u)})$ is included in the support of the filter given observations up to time $t_{i_j(\u)}$ is at least
$1-\left\|\mtx{H}\right\|_{\infty}\cdot \frac{d_o}{2\epsilon}\cdot \frac{\epsilon}{d_o\left\|\mtx{H}\right\|_{\infty}}=\frac{1}{2}$, and since $\u(t_{i_j(\u)})$ is included in the support, the stated result follows.
\end{proof}

\subsection{Characterisation of the support of the smoother of the geometric model as time tends to infinity}\label{secgeometricupperbounds}
In Section \ref{seclowerboundsgeometricsmoother} we have shown that for the geometric model, in the case of uniform observation errors in the interval $[-\epsilon,\epsilon]$, the expected value of the diameter of the support of smoothing distribution does not go to zero, but instead stays above $c(\u) h \epsilon$ for some constant $c(\u)$ only depending on the model parameters and the initial point $\u$. Let 
\begin{equation}\label{geoUuepsiondefeq}
U(\u,\epsilon):=\{\v\in \R^3: v_1=u_1, v_3=u_3, |v_2-u_2|< 2\epsilon\},
\end{equation}
which we will call the \emph{$2\epsilon$-cropped leaf set} of $\u$, a small line segment in the $u_2$ direction centered at $\u$. Our main result in this section characterises the support of the smoothing distribution by showing that it concentrates around the leaf set as the number of observations tends to infinity.

\begin{theorem}[Characterisation of the limit of the support of the smoother]\label{thmcharacterisationgeo}
Suppose that the observation matrix $\mtx{H}=\mtx{I}_{d\times d}$ (the identity matrix). Let $S_{k}$ be the support of the smoothing distribution of the geometric model based on the observations $Y_0, Y_1,\ldots, Y_k$. Then there are some positive constants $h_{\max}$ and $\epsilon_{\max}$ such that for any $0< \epsilon\le \epsilon_{\max}$,  $0< h\le h_{\max}$, for Lebesgue almost every $\u\in \Lambdageo$,
\begin{equation}\label{eqSkepsilonchar}
\sup_{\v\in S_k} d(\v, U(\u,\epsilon) )\to 0 \text{ as }k\to \infty \text{ almost surely in the observations},
\end{equation}
where $ d(\v, U(\u,\epsilon) ):=\inf_{\w\in U(\u,\epsilon)}\|\v-\w\|$. 
\end{theorem}
\begin{rem}
In Lemma \ref{lemmasuppunif} of Section \ref{seclowerbounds}, we prove a more precise formulation of the probability that a point $\v$ is included in the support of the smoother. Using that formulation, it is possible to show that every point in the $2\epsilon$-cropped leaf set $U(\u,\epsilon)$ have a positive probability of being included in the support of the smoother of the geometric model.
\end{rem}
The proof of this theorem is based on a few preliminary definitions and results. Since the support of smoother, $S_k$, are compact sets, if we let $S_{\infty}:=\cap_{k=1}^{\infty}S_k$, then one can show that the statement of Theorem \ref{thmcharacterisationgeo} is equivalent to showing that
\begin{equation}\label{eqSkintersection}S_{\infty} \subset \ol{U(\u,\epsilon)} \text{ almost surely},\end{equation}
where $\ol{U(\u,\epsilon)}$ denotes the closure of the $2\epsilon$-cropped leaf set $U(\u,\epsilon)$. Indeed, the fact that \eqref{eqSkepsilonchar} implies \eqref{eqSkintersection} is immediate. In the other direction, suppose that \eqref{eqSkintersection} holds but \eqref{eqSkepsilonchar} does not hold, then there is sequence of indices $i_{1}<i_{2}<\ldots$, a sequence of points $\v^{(i_1)}\in S_{i_1,\epsilon}, \v^{(i_2)}\in S_{i_2,\epsilon},\ldots$, and a positive constant $c>0$ such that  $d(\v^{(i_j)}, U(\u,\epsilon)>c$ for every $j\ge 1$. Due to the fact that $S_k$ are compact sets, and $S_k\subset S_{l}$ for $l<k$, we can see that the sequence $\v^{(i_j)}$ has at least one limiting point $\v^*$, which is in $S_{\infty}$, and by continuity of the distance function, satisfies that $d(\v^*,U(\u,\epsilon))\ge c$, contradicting \eqref{eqSkintersection}.

The next lemma establishes a useful expansion property of the return map $f$.
\begin{lem}\label{lemmafexpansion}
For every $a,b\in [-1/2,1/2]$ with $|a-b|\le 0.1$, we have \[|f(a)-f(b)|\ge \sqrt{2} |a-b|.\]
\end{lem}
\begin{proof}
If both $a$ and $b$ have the same sign, then this follows from the fact that $f'(x)>\sqrt{2}$ for every $x\in [-1/2,1/2]\setminus 0$. If they have different sign, then by \eqref{feq} and Assumption \ref{assalphabetatheta},
$|f(a)-f(b)| \ge 1-2\theta \cdot (0.1)^{\alpha}\ge 1-4 \cdot (0.1)^{1/\sqrt{2}}>0.1\sqrt{2}$, so the stated result holds.
\end{proof}
In order to fully exploit this expansion property, we will need to assume that the path $\u(t):=\Psi^{\mathrm{geo}}_t(\u)$ from the initial point $\u$ crosses $S$ infinitely many times, that is,
\begin{equation}\label{ucrossesSinfinitelyeq}
\u(t)\in S \text{ for infinitely many }t\ge 0.
\end{equation}
Based on the definition of the model, it is not difficult to show that this assumption is satisfied for Lebesgue-almost every $\u\in \Lambdageo$. So for the purpose of proving Theorem \ref{thmcharacterisationgeo}, for the rest of this section, we are going to assume that \eqref{ucrossesSinfinitelyeq} holds.

Let $B_{2\epsilon}^{\infty}(\u):=\{\v \in \Lambdageo: \|\v-\u\|_{\infty}\le 2\epsilon\}$, $t_{\max}:=\frac{0.05}{\vmaxgeo}$ (see \eqref{vmaxgeodefeq}), and define the \emph{time-shifted $2\epsilon$-cropped leaf set} of $\u$ as
\begin{align}
\nonumber W(\u,\epsilon,t_{\max}):=\big\{&\v \in B_{2\epsilon}^{\infty}(\u) \text{ for which there is a }\w\in U(\u,\epsilon) \text{ and }t\in [-t_{\max},t_{\max}]\\
&\text{ such that either } t\ge 0\text{ and }\w(t)=\v \text{ or }t< 0\text{ and }\v(-t)=\w\big\}.\label{eqtimeshiftedleaf}
\end{align}
Based on the expansion property of the return map $f$, the following lemma shows that only the points in the time-shifted $2\epsilon$-cropped leaf set $W(\u,\epsilon,t_{\max})$ can be included in $S_{\infty}$.
\begin{lem}\label{lemSinftyinWuepsilon}
Let $h_{\max}:=\frac{1}{80}\frac{\vmingeo}{(\vmaxgeo)^2}$ and $\epsilon_{\max}:=\frac{1}{80}\frac{\vmingeo}{\vmaxgeo}$, then for any $0< \epsilon\le \epsilon_{\max}$,  $0< h\le h_{\max}$, we have $S_{\infty}\subset W(\u,\epsilon,t_{\max})$.
\end{lem}
\begin{proof}
First note that since the maximum speed of the dynamics is bounded by $\vmaxgeo$ (see \eqref{vmaxgeodefeq}), we have that
\begin{equation}\label{maxspeednotinveq}
\v \notin S_{\infty} \text{ if }\|\v(t)-\u(t)\|_{\infty}>2\epsilon+2h\vmaxgeo \text{ for some }t\ge 0,
\end{equation}
since otherwise there would certainly exist some $k\in \N$ such that $\|\v(t_k)-\u(t_k)\|_{\infty}>2\epsilon$.

Since $S_{\infty}\subset B_{2\epsilon}^{\infty}(\u)$, we only need to check the points $\v\in B_{2\epsilon}^{\infty}(\u)\setminus W(\u,\epsilon,t_{\max})$. 
Note that since we have assumed in \eqref{ucrossesSinfinitelyeq} that $\u(t)$ crosses $S$ infinitely often, we can also assume without loss of generality that $\v(t)$ crosses $S$ infinitely often, otherwise $\v\notin S_{\infty}$ by \eqref{maxspeednotinveq}.

Suppose first that $\u, \v\in W^{S}_{0.1}$ with $u_3\ge 1$ and  $v_3<1$ (thus $\u$ is above $S$ and $\v$ is below $S$ on Figure \ref{geometric:sub1}) . Then define $\u'$ as the first intersection of $\u(t)$ and $S$ for $t\ge 0$, and let $\v':=O(\v)$ (the origin of $v$ on $S$, see \eqref{originofveq}).

Now we compare the first coordinates $u'_1$ and $v'_1$. If $u'_1=v'_1$, and $\v\in S_{\infty}$, then from the definition of the process, we can see that there must exist a point $\w \in U(\u)$ and a constant $s>0$ such that $\w(s)=\v$. Moreover, from \eqref{vmingeodefeq} it follows that $s\le \frac{2\epsilon_{\max}}{\vmingeo}\le t_{\max}$, so therefore $\v$ must be in $W(\u,\epsilon,t_{\max})$, which we do not need to check. 

Alternatively, if $u'_1\ne v'_1$, then by Lemma \ref{lemmafexpansion}, after sufficient amount of returns, the first coordinates will satisfy that $|f^{(k)}(u'_1)-f^{(k)}(v'_1)|>0.1$ (here $f^{(k)}(x)$ denotes the $k$ times composition of $f$ with itself). However, if $\|\v(t)-\u(t)\|_{\infty}\le 2\epsilon+2h\vmaxgeo$ for every $t\ge 0$, then by \eqref{vmingeodefeq} and \eqref{vmaxgeodefeq} we know that the return points on $S$ cannot be further away than 
\[(2\epsilon+2h\vmaxgeo)+\frac{(2\epsilon+2h\vmaxgeo)}{\vmingeo}\cdot \vmaxgeo<0.1.\]
Thus  $\|\v(t)-\u(t)\|_{\infty}\le 2\epsilon+2h\vmaxgeo$ cannot hold for every $t\ge 0$, and by \eqref{maxspeednotinveq},  $\v\notin S_{\infty}$.

In the case when $\u, \v\in W^{S}_{0.1}$ with $u_3< 1$ and  $v_3\ge 1$, we define $\u':=O(\u)$ and $\v'$ as the first intersection of $\v(t)$ and $S$ for $t\ge 0$. Finally, in every other situation we define $\u':=O(\u)$ and $\v':=O(\v)$. The rest of the argument is the same as in the case we have considered above.
\end{proof}

Now we are ready to prove the main result of this section.
\begin{proof}[Proof of Theorem \ref{thmcharacterisationgeo}]
Based on \eqref{eqSkintersection} and Lemma \ref{lemSinftyinWuepsilon}, it suffices to check points $\v$ in the $2\epsilon$-cropped time-shifted leaf set $W(\u,\epsilon,t_{\max})$ (see \eqref{eqtimeshiftedleaf}).  For such points, let $\Delta(\v)$ denote the value of $t$ in the definition \eqref{eqtimeshiftedleaf}, this is the \emph{time shift} of $\v$, satisfying that $|\Delta(\v)|\le t_{\max}$. For $0< s\le t_{\max}$, let us define the \emph{restrictions of the time-shifted leaf set} as
\begin{align*}
W_+(\u,\epsilon,s)&:=\{\v\in W(\u,\epsilon,t_{\max}): \Delta(\v)>s\}, \quad \text{ and }\\
W_-(\u,\epsilon,s)&:=\{\v\in W(\u,\epsilon,t_{\max}): \Delta(\v)<-s\}.
\end{align*}
For $h\le h_{\max}$ (defined as in Lemma \ref{lemSinftyinWuepsilon}), one can see that there are going to be an infinite sequence of observation times $t_{i_1},t_{i_2},\ldots$ such that $\u(t_{i_j})\in W^S_{0.05}$ (see \eqref{WSsdefeq}). At these time points, using \eqref{vmingeodefeq}, for $0< s\le t_{\max}$, we have
\[\inf_{\v\in W_+(\u,\epsilon,s)}v_3(t_{i_j})\ge u_3(t_{i_j})+s\vmingeo\text{ and }\sup_{\v\in W_-(\u,\epsilon,s)}v_3(t_{i_j})\le u_3(t_{i_j})-s\vmingeo.\]
Let $Z^{3}_{i_j}$ denote the third component of the observation noise at time $t_{i_j}$, then if $Z^{3}_{i_j}< -\epsilon+s\vmingeo$, then none of the points in the restriction of the $2\epsilon$-cropped time-shifted leaf set $W_+(\u,\epsilon,s)$ are in the limiting set $S_{\infty}$. This event has probability $s\vmingeo/2\epsilon>0$, and since there are infinitely many such indices $i_j$, and $(Z^{3}_{i_j})_{j\ge 1}$ are independent, therefore $W_+(\u,\epsilon,s)$ and $W_-(\u,\epsilon,s)$ are almost surely disjoint with $S_{\infty}$. Since we can write $W(\u,\epsilon,t_{\max})\setminus\ol{U(\u,\epsilon)}$ as a countable union 
\[W(\u,\epsilon,t_{\max})\setminus \ol{U(\u,\epsilon)}=\cup_{i\ge 1} \left(W_+(\u,\epsilon,t_{\max}/i)\cup W_-(\u,\epsilon,t_{\max}/i)\right),\]
therefore almost surely only the points $\v$ in $\ol{U(\u,\epsilon)}$ can be included in the limiting set $S_{\infty}$, and thus \eqref{eqSkepsilonchar} follows via \eqref{eqSkintersection}.
\end{proof}

\section{Lower bounds for a class of chaotic dynamical systems}\label{seclowerbounds}
In Section \ref{seclowerboundsgeometricsmoother}, we have given lower bounds for the smoother and the filter of the geometric model. In this section, we will extend such results to a class of chaotic dynamical systems satisfying some appropriate assumptions. We treat the cases of both uniform and Gaussian error distributions. Our results are organised into four subsections. In Section \ref{seclowerboundsuniform}, we consider lower bounds on the diameter of the support of the smoother and the filter under uniform error distributions. This is followed by Section \ref{seclowerboundsgaussian}, where we consider bounds for Gaussian error distributions. Finally, Section \ref{seclowerboundssimulations} gives some numerical simulations that seem to indicate the validity of the assumptions of the previous three sections for the Lorenz 63' and Lorenz 96' models.
\subsection{Lower bounds for uniform noise}\label{seclowerboundsuniform}
In this section, we will first consider the support of the smoothing and filtering distributions (\eqref{eqmusm} and \eqref{eqmufi}) when we have observation matrix $\mtx{H}$, and the noise variables $(\vct{Z}_i)_{i\ge 0}$ are i.i.d., uniformly distributed in $[-\epsilon,\epsilon]^{d_o}$. 
The $L_1$ norm of a vector $\v\in \R^d$ is defined as $\|\v\|_1:=\sum_{i=1}^d |v_i|$. For a matrix $\mtx{M}\in \R^{d_1\times d_2}$, we let its $L_1$ norm be the induced norm $\|\mtx{M}\|_{1}:=\sup_{\v\in \R^{d_2}, \|\v\|_1\le 1}\|\mtx{M}\v\|_1$. The following lemma is a key tool in this section.
\begin{lem}[Bounding the probability that a point is in the support]\label{lemmasuppunif}
Let $\vct{Y}_i=\mtx{H}\u(t_i)+\vct{Z}_i$ be the noisy observations at time points $0\le t_i\le t_k$ obtained from \eqref{diffeqgeneralform} started at some initial point $\u\in \BR$. Suppose that observation errors are uniformly distributed in $[-\epsilon,\epsilon]^{d_o}$. Suppose that $\v\in \BR$ is a fixed point, and the prior $q$ satisfies that $q(\v)>0$ and $q(\u)>0$. Let $S_k$ denote the support of the smoothing distribution $\musm(\cdot |\vct{Y}_0,\ldots, \vct{Y}_k)$. Then the probability that $\v$ is included in the support of the smoothing distribution is given as
\begin{align}\label{probvinsupporteq}\PP(\v\in S_k|\u)=\prod_{i=0}^k\prod_{j=1}^{d_o} \left( 1-\frac{|(\mtx{H}\u(t_i))_j- (\mtx{H}\v(t_i))_j)|}{2\epsilon}\right)_+.
\end{align}
Let $D_k^{(1)}(\u,\v):=\sum_{i=0}^k \|\v(t_i)-\u(t_i)\|_1$, and $M_k(\u,\v):=\max_{0\le i\le k} \|\mtx{H}\v(t_i)-\mtx{H}\u(t_i)\|_{\infty}$. Then the probability of the inclusion can be lower bounded as
\begin{equation}
\label{probvinsupportlowerboundeq}\PP(\v\in S_k|\u)\ge\exp\left(-\frac{D_k^{(1)}(\u,\v)\|\mtx{H}\|_1}{\epsilon}\right) \text{ for } \epsilon\ge M_k(\u,\v).
\end{equation}
Moreover, if $d_o=d$ and $\mtx{H}$ is not singular, then we have the upper bound
\begin{equation}\label{probvinsupportupperboundeq}
\PP(\v\in S_k|\u)\le \exp\left(-\frac{D_k^{(1)}(\u,\v)}{2 \epsilon \|\mtx{H}^{-1}\|_1}\right) \text{ for any }\epsilon>0.\\
\end{equation}
\end{lem}
\begin{proof}
Let $w\in \R$, and $W$ be uniformly distributed in $w-\epsilon, w+\epsilon$. Then the probability that another point $r\in \R$ is less than $\epsilon$ away from $W$ is
\[\PP(|r-W|\le \epsilon)=\frac{(2\epsilon-|r-w|)_+}{2\epsilon}=\left( 1-\frac{|r-w|}{2\epsilon}\right)_+.\]
Using this and the independence of the components of the noise vectors $(\vct{Z}_i)_{0\le i\le k}$, we have
\begin{align*}\PP(\v\in S_k|\u)&=\PP\left(\|\v(t_i)-\vct{Y}_i\|_{\infty}\le \epsilon \text{ for }0\le i\le k |\u\right)\\
&=\E\left(\prod_{i=0}^k \II\left[\|\mtx{H}\v(t_i)-(\mtx{H}\u(t_i)+\vct{Z}_i)\|_{\infty}\le \epsilon\right] | \u\right)\\
&=\prod_{i=0}^k\prod_{j=1}^{d_o} \left( 1-\frac{|(\mtx{H}\u(t_i))_j- (\mtx{H}\v(t_i))_j)|}{2\epsilon}\right)_+.
\end{align*}
The upper bound \eqref{probvinsupportupperboundeq} follows by taking the logarithm of both sides and using the inequality $\log((1-x)_+)\le -x$ for $x>0$. The lower bound \eqref{probvinsupportlowerboundeq} follows from the fact that $1-x\ge \exp(-2x)$ for $0\le x\le \frac{1}{2}$.
\end{proof}
A consequence of this lemma is that if $\sup_{k\in\N}D_k^{(1)}(\u,\v)=\infty$, and $\mtx{H}$ is not singular, then the probability that $\v$ is in the support of the smoother tends to $0$ as $k\to \infty$. Conversely, if $\sup_{k\in\N}D_k^{(1)}(\u,\v)<\infty$, then for $\epsilon$ sufficiently large (larger than 
$\sup_{k\in \N}M_k(\u,\v))$, the probability that $\v$ is included in the support of the smoother is  lower bounded by $\exp\left(-\frac{1}{\epsilon}\sup_{k\in\N}D_k^{(1)}(\u,\v)\right)$, independently of $k$.
Due to this property, we have found that the following assumption is useful for establishing lower bounds on the diameter of the smoothing distribution.
\begin{ass}\label{assuniflow}
Suppose that there is a set of points $U(\u)\subset \BR$, called the \emph{leaf set} of $\u$, such that
$\{\|\v-\u\|_1: \v\in U(\u) \}=[0, d_{\max}(\u)]$ for some $d_{\max}(\u)>0$, and that there is a finite constant $C_U^{(1)}(\u)$ such that for every $\v\in U(\u)$,
\begin{equation}\label{assunifloweq}
\sup_{k\in\N}D_k^{(1)}(\u,\v)\le  C_U^{(1)}(\u) \|\v-\u\|_1.
\end{equation}
\end{ass}
The assumption essentially means that there exists a curve $U(\u)$ containing $\u$ such that for every point $\v\in U(\u)$, the distance between $\u(t)$ and $\v(t)$ tends to 0 as $t$ tends to infinity (at a sufficiently quick rate). This concept of leaf set is similar to the concept of a leaf of a foliation used in \cite{Tucker2002}, see also \cite{Viana2000}. Note that it is also similar to the concept of stable set (also called local stable manifold) used in the theory of dynamical systems (see e.g. page 18 of \cite{Shubglobalstability}).  Note that in this assumption, $h$ is fixed and does not tends to zero (and the constant $C_U^{(1)}(\u)$ depends on $h$). In the case of the geometric model,  Theorem \ref{thmlowerboundsmoother1} has shown the existence of a leaf set in $U(\u)$, and based on \eqref{eqpathdistance}, one can see that the condition \eqref{assunifloweq} of the above assumption is satisfied for any $h>0$.

The following theorem gives a lower bound for the smoother based on the above assumption. Section \ref{seclowerboundssimulations} includes numerical tests of this assumption for the Lorenz 63' and 96'  models. 
\begin{theorem}[Lower bound on the diameter of the support of the smoother]\label{thmlowerboundsmoother}
Suppose that Assumption \ref{assuniflow} holds, we have observation matrix $\mtx{H}$, and that the observation noise is uniformly distributed in $[-\epsilon,\epsilon]^{d_o}$. Suppose that the prior $q$ satisfies that $q(\v)>0$ for every $\v\in U(\u)$. Then for $\epsilon\le d_{\max}(\u) \cdot C_{U}^{(1)}(\u)\cdot \|\mtx{H}\|_1$, we have
\begin{equation}\label{diamlowboundunifeq}
\E(\mathrm{diam}_1\mathrm{ supp} \musm(\cdot |\vct{Y}_0,\ldots, \vct{Y}_k)|\u)\ge \frac{1}{e}\cdot \frac{\epsilon}{C_{U}^{(1)}(\u) \cdot \|\mtx{H}\|_1},
\end{equation}
where $\mathrm{diam}_1\mathrm{ supp}$ denotes diameter of the support with respect to the $L_1$ norm.
\end{theorem}
\begin{proof}
Let $\v\in U(\u)$, then based on Lemma \ref{lemmasuppunif} and Assumption \ref{assuniflow}, we have that for $\epsilon\ge M_k(\u,\v)$, 
\begin{equation*}\E(\mathrm{diam}_1\mathrm{ supp} \musm(\cdot |\vct{Y}_0,\ldots, \vct{Y}_k)|\u)\ge \exp\left(-\frac{C_U^{(1)}(\u)\|\mtx{H}\|_1 \|\v-\u\|_{1}}{\epsilon}\right)\cdot \|\v-\u\|_1.
\end{equation*}
Using the fact that $\epsilon \le d_{\max}(\u)C_U^{(1)}(\u)\|\mtx{H}\|_1$, we can choose $\v\in U(\u)$ such that $\|\v-\u\|_1=\frac{\epsilon}{C_U^{(1)}(\u)\|\mtx{H}\|_1}$. For this choice of $\v$, we have 
\[\epsilon=\|\mtx{H}\|_1 C_U^{(1)}(\u)\|\|\v-\u\|_1\ge\|\mtx{H}\|_1  D_k^{(1)}(\u,\v) \ge M_k(\u,\v),\] so the result follows by the above inequality.
\end{proof}

After the smoother, now we show some lower bounds for the filter that are analogous to those we have obtained for the geometric model (see Theorem \ref{thmgeomodelfilter}). We use the following assumption.
\begin{ass}\label{assuniflowfilter}
Suppose that for the initial position $\u\in \R^d$, there are sets $\tU(\u,k)\subset \R^d$, called \emph{anti-leaf} sets, such that $\{\|\v(t_k)-\u(t_k)\|_1: \v\in \tU(\u,k)\}=[0,\tilde{d}_{\max}(\u,k)]$, and for every point $\v\in \tU(\u,k)$, we have
\begin{align}\label{eqasslowfilter1}
D^{(1)}_k(\u,\v)\le C_{\tU}(\u,k) \cdot \|\v(t_k)-\u(t_k)\|_1,
\end{align}
for some constants  $C_{\tU}(\u,k)$ and $\tilde{d}_{\max}(\u,k)$. Moreover, suppose that there are infinitely many indices $i_1<i_2<\ldots$ such that for every $i_j$, we have 
\begin{equation}\label{eqasslowfilter2}
 C_{\tU}(\u,i_j)\le C_{\tU}(\u) ,  \quad \text{ and }\quad \tilde{d}_{\max}(\u,i_j)\ge \tilde{d}_{\max}(\u),
\end{equation}
for some constants $C_{\tU}(\u)<\infty$ and  $\tilde{d}_{\max}(\u)>0$.
\end{ass}
This assumption essentially means that there are anti-leaf sets $\tU(\u,k)$, which are curves containing $\u$ such that for points $\v\in \tU(\u,k)$,  $\|\u(t)-\v(t)\|$ is typically growing in $t$ up to time point $t_k$. They behave in the exact opposite way when compared to leaf sets, hence the name anti-leaf set. This is a rather natural assumption if the system behaves chaotically, and the path of almost every two nearby points get far away eventually. The definition is somewhat similar to the definition of unstable sets (also called local unstable manifolds) used in the theory of dynamical systems (see e.g. page 18 of \cite{Shubglobalstability}).In the case of the geometric model, Theorem \ref{thmgeoantileafsets} has established the existence of anti-leaf sets, which also satisfy conditions \eqref{eqasslowfilter1} and \eqref{eqasslowfilter2} of the above assumption. A numerical test of this assumption for the Lorenz 63' model is included in Section \ref{seclowerboundssimulations}. 
\begin{theorem}[Lower bound on the diameter of the support of the filter]\label{thmlowerbndfilter}
Under Assumption \ref{assuniflowfilter}, if we have observation matrix $\mtx{H}$, the observation noise is uniformly distributed in $[-\epsilon,\epsilon]^{d_o}$, and the prior satisfies that $q(\v)>0$ for every $\v\in \cup_{k\in \N} \tU(\u,k)$, then for any $\epsilon\le \tilde{d}_{\max}(\u) C_{\tilde{U}}(\u) \|\mtx{H}\|_1$, and any $j\ge 1$, we have 
\[\E\left[\mathrm{diam}_1 \supp\mufi(\cdot|\vct{Y}_0,\ldots, \vct{Y}_{i_j})|\u\right]\ge \frac{1}{e}\cdot \frac{\epsilon}{C_{\tU}(\u) \|\mtx{H}\|_1}.\]
\end{theorem}
\begin{proof}
The proof is analogous to the proof of Theorem \ref{thmlowerboundsmoother}. We choose $\v\in  \tU(\u,i_j)$ such that $\|\v(t_{i_j})-\u(t_{i_j})\|_1=\frac{\epsilon}{c_U(\u)  \|\mtx{H}\|_1}$ (by the assumption on $\epsilon$, this is possible), and the result follows from inequality \eqref{probvinsupportlowerboundeq} of Lemma \ref{lemmasuppunif}.
\end{proof}

For the geometric model, we have been able to explicitly characterise the limit of the support of the smoother as the number of observations tends to infinity (see Theorem \ref{thmcharacterisationgeo}). It is possible to generalise this result to other chaotic dynamical systems satisfying the following assumptions. 
\begin{ass}\label{assunifchar}
For any $\epsilon>0$, $s>0$, let
\begin{align} U(\u,\epsilon):=&\left\{\v\in \BR: \sum_{i=0}^{\infty} \|\v(t_i)-\u(t_i)\|_{\infty}<\infty, \sup_{i\ge 0}\|\v(t_i)-\u(t_i)\|_{\infty}<2\epsilon\right\}\label{Uuepsdefeq}\\
W(\u,\epsilon,s):=&\{\v\in \BR: \|\u-\v\|_{\infty}<2\epsilon, \text{ and there is some } \w\in U(\u,\epsilon) \text{ and }t\in [-s,s]\nonumber\\
&\text{ such that either }t\ge 0\text{  and  }\w(t)=\v \text{ or }t<0 \text{ and }\v(t)=\w\}.\label{Wuepstmaxdefeq}
\end{align}
We call $U(\u,\epsilon)$ as the \emph{$2\epsilon$-cropped leaf set of $\u$}, and $W(\u,\epsilon,s)$ the time shifted $2\epsilon$-cropped leaf set of $\u$. Suppose that there is some constant $v_{\min}(\u)>0$ such that 
\begin{equation}\label{eqdotutivmin}
\left\|\frac{d}{dt}\u(t_i)\right\|_{\infty}\ge v_{\min}(\u) \text{ for infinitely many }i\in \N.\end{equation}
Suppose that there is sequence of reals $(\rho_i(\u,\epsilon))_{i\in N}$ such that $\lim_{k\to \infty}\rho_k=0$, and for any $\v\in U(\u,\epsilon)$, we have 
\begin{equation}\label{eqsumrhok}\sum_{i=k}^{\infty} \|\v(t_i)-\u(t_i)\|_{\infty}\le \rho_k(\u,\epsilon).\end{equation}
Suppose that there is a constant $t_{\max}(\u,\epsilon)\in (0, \frac{v_{\min}(\u)}{6a_{\max} v_{\max}})$
such that for any $\v\notin W(\u,\epsilon,t_{\max}(\u,\epsilon))$, we have
\begin{equation}\label{eqnotinWuepstmax}\|\v(t_i)-\u(t_i)\|_{\infty}>2\epsilon\text{ for some }i\ge 0.
\end{equation}
\end{ass}
In the above assumption, $v_{\max}$ and $a_{\max}$ are defined according to \eqref{eqvmax} and \eqref{eqamax}. This assumption contains the essential properties of the dynamics that were used in the proof of Theorem \ref{thmcharacterisationgeo} for the geometric model.
In that case, the $2\epsilon$-cropped leaf set $U(\u,\epsilon)$ was defined in equation \eqref{geoUuepsiondefeq}, condition \eqref{eqdotutivmin} was implied by \eqref{vmingeodefeq},
and the condition \eqref{eqnotinWuepstmax} was proven in Lemma \ref{lemSinftyinWuepsilon}. 

The following result shows that under Assumption \ref{assunifchar}, as the number of observations tends to infinity, the support of the smoother gets concentrated around the $2\epsilon$-cropped leaf set $U(\u,\epsilon)$.
\begin{theorem}[Characterisation of the limit of the support of the smoother]\label{thmcharacterizationsmoother}
Suppose that Assumption \ref{assunifchar} holds, and that $q(\v)>0$ for every $\v\in U(\u,\epsilon)$.  Suppose that the observation matrix $\mtx{H}=\mtx{I}_{d\times d}$ and the observation errors are uniformly distributed on $[-\epsilon,\epsilon]^d$. Then
\begin{equation}\label{eqcharsuppsm1}\sup_{\v\in \mathrm{supp}\,\musm(\cdot|\vct{Y}_0,\ldots, \vct{Y}_k)} d(\v,   U(\u,\epsilon))\to 0 \text{ as }k\to \infty \text{ almost surely in the observations,}\end{equation}
where $d(\v,  U(\u,\epsilon))=\inf_{\w\in U(\u,\epsilon)}\|\v-\w\|$. Moreover, for every point $\v \in U(\u,\epsilon)$, we have 
\[\inf_{k\in \N}\PP(\left. \v\in \mathrm{supp}\,\musm(\cdot|\vct{Y}_0,\ldots, \vct{Y}_k)\right| \u)>0.\]
\end{theorem}
The proof of this theorem is similar to the proof of Theorem \ref{thmcharacterisationgeo}. It is included in Section \ref{secproofthmcharacterizationsmoother} of the Appendix.

\subsection{Lower bounds for Gaussian noise}\label{seclowerboundsgaussian}
In this section we generalise the results of the previous section to Gaussian noise. In this case, the quantity of interest will be the diameter of the support of the set of points whose likelihood is no less that $1/e$ times the likelihood of the true position. The following lemma is a key tool in this section.
\begin{lem}[Bounding the probability that a point has large likelihood]\label{lemmasuppgauss}
Let $\vct{Y}_i=\mtx{H}\u(t_i)+\vct{Z}_i$ be the noisy observations at time points $0\le t_i\le t_k$ obtained from \eqref{diffeqgeneralform} started at some initial point $\u\in \BR$. Suppose that observation errors satisfy that $\vct{Z}_i/\epsilon$ has $d_o$ dimensional standard Gaussian distribution for every $0\le i\le k$. Let 
\begin{align}
\label{Dk2uvdefeq}D_k^{(2)}(\u,\v)&:=\sum_{i=0}^k \|\Psi_{t_i}(\v)-\Psi_{t_i}(\u)\|_2^2, \text{ and }\\
D_k^{(2)}(\u,\v,\mtx{H})&:=\sum_{i=0}^k \|\mtx{H}\Psi_{t_i}(\v)-\mtx{H}\Psi_{t_i}(\u)\|_2^2\le \|\mtx{H}\|_2^2 D_k^{(2)}(\u,\v).
\end{align}
Suppose that $q(\u)>0$, then for any $\v\in \BR$,  we have
\begin{align}\label{probvinsupportgausseq}
&\PP\left[\left.\frac{\musm(\v |Y_0,\ldots, Y_k)}{\musm(\u |Y_0,\ldots, Y_k)}\ge \frac{q(\v)}{q(\u)}\cdot  \exp\left(-\frac{D_k^{(2)}(\u,\v,\mtx{H})}{2\epsilon^2}\right)\right|\u\right]\ge \frac{1}{2}, \text{ and}\\
\label{probvinsupportgaussfieq}
&\PP\left[\left.\frac{\mufi(\v(t_k) |Y_0,\ldots, Y_k)}{\mufi(\u(t_k) |Y_0,\ldots, Y_k)}\ge \frac{q(\v)}{q(\u)}\cdot  \frac{ \mathrm{det}(\mtx{J}\Psi_{t_k}(\u))}{ \mathrm{det}(\mtx{J}\Psi_{t_k}(\v))}\cdot\exp\left(-\frac{D_k^{(2)}(\u,\v,\mtx{H})}{2\epsilon^2}\right)\right|\u\right]\ge \frac{1}{2}.
\end{align}
\end{lem}
\begin{proof}
By the definition of the smoothing distribution $\musm$, we have
\begin{align*}&\frac{\musm(\v |Y_0,\ldots, Y_k)}{\musm(\u |Y_0,\ldots, Y_k)}=\frac{q(\v)}{q(\u)}\cdot \exp\left( -\frac{1}{2\epsilon^2}\cdot \sum_{i=0}^k \left(\|\mtx{H}\v(t_i)-Y_i\|^2-\|\mtx{H}\u(t_i)-Y_i\|^2\right)\right)\\
&=\frac{q(\v)}{q(\u)}\cdot \exp\left( -\frac{1}{2\epsilon^2}\cdot\sum_{i=0}^k \left(\|\mtx{H}\v(t_i)-\mtx{H}\u(t_i)\|^2+2\left<\mtx{H}\v(t_i)-\mtx{H}\u(t_i),\vct{Z}_i\right>\right)\right)\\
&=\frac{q(\v)}{q(\u)}\cdot \exp\left( -\frac{1}{2\epsilon^2}\cdot \left(D_k^{(2)}(\u,\v,\mtx{H})+2\sum_{i=0}^k \left<\mtx{H}\v(t_i)-\mtx{H}\u(t_i),\vct{Z}_i\right>\right)\right),
\end{align*}
and \eqref{probvinsupportgausseq} follows from the fact that $\sum_{i=0}^k \left<\mtx{H}\v(t_i)-\mtx{H}\u(t_i),\vct{Z}_i\right>$ is a Gaussian random variable with mean 0. The proof of \eqref{probvinsupportgaussfieq} is similar.
\end{proof}

\begin{ass}\label{assgausslow}
Suppose that there is a set of points $U(\u)\subset \BR$ called the \emph{leaf set} of $\u$ such that
$\{\|\v-\u\|_2: \v\in U(\u) \}=[0, d_{\max}^{(2)}(\u)]$ for some $d_{\max}^{(2)}(\u)>0$, and that there is a finite constant $C_U^{(2)}(\u)$ such that for every $\v\in U(\u)$,
\begin{equation}\label{assgaussloweq}
\sup_{k\in \N}D_k^{(2)}(\u,\v)\le  C_U^{(2)}(\u) \|\v-\u\|_2^2.
\end{equation}
\end{ass}
Similarly to Assumption \ref{assuniflow}, this assumption essentially means that there exists a leaf set $U(\u)$, which is a curve containing $\u$ such that for every point $\v\in U(\u)$, $\|\u(t)-\v(t)\|\to 0$ as $t\to \infty$  (at a sufficiently quick rate). In the case of the geometric model, based on \eqref{eqpathdistance}, one can see that the leaf set $U(\u)$ defined in Theorem \ref{thmlowerboundsmoother1} satisfies the condition \eqref{assgaussloweq} of the above assumption  for any $h>0$. A numerical test of this assumption is included for the Lorenz 63' and 96' models in Section \ref{seclowerboundssimulations}. The following theorem lower bounds the diameter of the set of points whose likelihood is not much smaller than the likelihood of the true initial position.
\begin{theorem}[Lower bound on the diameter of the set of high likelihood for the smoother]\label{thmgausslowerboundsmoother}
Suppose that Assumption \ref{assgausslow} holds, and that the density of the prior $q$ is continuous at the point $\u$, and $q(\v)>0$ for every $\v\in U(\u)$. Then for $\epsilon$ sufficiently small,
\begin{equation}\label{diamlowboundgausseq}
\E\left[\left.\mathrm{diam}_2\,\mathrm{supp}\left\{\v\in \R^d: \frac{\musm(\v |\vct{Y}_0,\ldots, \vct{Y}_k)}{\musm(\u |\vct{Y}_0,\ldots, \vct{Y}_k)}\ge \frac{1}{e}\right\} \right|\u\right]\ge \frac{1}{2\|\mtx{H}\|_2\sqrt{C_U^{(2)}(\u)}}\cdot \epsilon,
\end{equation}
where $\mathrm{diam}_2\,\mathrm{supp}$ denotes diameter of the support with respect to the Euclidean distance.
\end{theorem}
\begin{proof}
We choose $\v\in U(\u)$ such that $\|\u-\v\|_2=\frac{1}{\|\mtx{H}\|_2\sqrt{C_U^{(2)}(\u)}}\cdot \epsilon$, this is possible if $\epsilon\le d_{\max}^{(2)}(\u) \|\mtx{H}\|_2 \sqrt{C_U^{(2)}(\u)}$. By the continuity of $q$, we have $\frac{q(\v)}{q(\u)}\ge \frac{1}{\sqrt{e}}$ for $\epsilon$ sufficiently small, and the result follows from Lemma \ref{lemmasuppgauss}.
\end{proof}
We end this section by stating a similar result for the filtering distribution. We are going to use the following assumption.
\begin{ass}\label{assgausslowfilter}
Let $\u$ be the initial position, and for any $k\in \N$, define the sets 
\[S_{\J}(\u,k):=\{\v\in \R^d: \mathrm{det}(\mtx{J}\Psi_{t_k}(\v))=\mathrm{det}(\mtx{J}\Psi_{t_k}(\u))\},\]
where $\mtx{J}\Psi_{t_k}(\v)$ is the $d\times d$ Jacobian matrix (defined in Section \ref{SecPreliminaries}). 

Suppose that for the initial position $\u\in \R^d$, there are sets $\tU(\u,k)\subset S_{\J}(\u,k)$, called \emph{anti-leaf} sets, such that $\{\|\v(t_k)-\u(t_k)\|_2: \v\in \tU(\u,k)\}=[0,\tilde{d}_{\max}^{(2)}(\u,k)]$, and for every point $\v\in \tU(\u,k)$, we have
\begin{align*}
D_k^{(2)}(\u,\v)\le C_{\tU}^{(2)}(\u,k) \cdot \|\v(t_k)-\u(t_k)\|_2^2,
\end{align*}
for some constants  $C_{\tU}^{(2)}(\u,k)$ and $\tilde{d}_{\max}^{(2)}(\u,k)$. Moreover, suppose that there are infinitely many indices $i_1<i_2<\ldots$ such that for every $i_j$, we have 
\[ C_{\tU}^{(2)}(\u,t_{i_j})\le C_{\tU}^{(2)}(\u),  \quad \text{ and }\quad \tilde{d}_{\max}^{(2)}(\u,t_{i_j})\ge \tilde{d}_{\max}^{(2)}(\u),\]
for some constants $C_{\tU}^{(2)}(\u)<\infty$ and  $\tilde{d}_{\max}^{(2)}(\u)>0$.
\end{ass}
This assumption is similar to Assumption \ref{assuniflowfilter} that we had in the uniform case. However, it also includes the restriction that $\tU(\u,k)\subset S_{\J}(\u,k)$, i.e. the anti-leaf sets should be included in the level set of the determinant of the Jacobian. This is necessary because the determinant of the Jacobian can have a large influence on the likelihood of the filter. Note that in general the set $S_{\J}(\u,k)$ is a $d-1$ dimensional manifold which satisfies that its surface is perpendicular to the gradient $\grad\mathrm{det}(\mtx{J}\Psi_{t_k}(\v))$ at each point $\v$.  The following theorem shows a lower bound for the filter under this assumption.
\begin{theorem}[Lower bound on the diameter of the set of high likelihood for the filter]
Suppose that Assumption \ref{assgausslowfilter} holds, and that the density of the prior $q$ is continuous at the point $\u$, and $q(\v)>0$ for every $\v\in \cup_{k\in \N} \tU(\u,k)$. Then for $\epsilon$ sufficiently small, for any $j\ge 1$, we have
\begin{equation}\label{diamlowboundgaussfiltereq}
\E\left[\left.\mathrm{diam}_2\,\mathrm{supp}\left\{\v\in \R^d: \frac{\mufi(\v |\vct{Y}_0,\ldots, \vct{Y}_{i_j})}{\mufi(\u(t_{i_j}) |\vct{Y}_0,\ldots, \vct{Y}_{i_j})}\ge \frac{1}{e}\right\} \right|\u\right]\ge \frac{1}{2\|\mtx{H}\|_2\sqrt{c_U^{(2)}(\u)}}\cdot \epsilon,
\end{equation}
where $\mathrm{diam}_2\,\mathrm{supp}$ denotes diameter of the support with respect to the Euclidean distance.
\end{theorem}
\begin{proof}
The proof is analogous to the proof of Theorem \ref{thmgausslowerboundsmoother}. We choose $\v\in  \tU(\u,i_j)$ such that $\|\v(t_{i_j})-\u(t_{i_j})\|_2=\frac{\epsilon}{\|\mtx{H}\|_2\sqrt{C_{\tU}^{(2)}(\u)}}$ (this is possible when $\epsilon\le \tilde{d}_{\max}^{(2)}(\u) \|\mtx{H}\|_2 \sqrt{  C_{\tU}^{(2)}(\u)}$), and the result follows from \eqref{probvinsupportgaussfieq} of Lemma \ref{lemmasuppgauss} and the fact that $q$ is continuous in $\vct{u}$.
\end{proof}

\subsection{Numerical illustration for the Lorenz 63' and Lorenz 96' models}\label{seclowerboundssimulations}
In this section, we present some numerical evidence that supports the assumptions we have made in Sections \ref{seclowerboundsuniform}-\ref{seclowerboundsgaussian}. First, we will treat the assumptions related to the smoother, and then the assumptions related to the filter.

\subsubsection{Assumptions for the smoother}
In the following figures, we will provide numerical evidence about the existence of the leaf set $U(\vct{u})$ with properties required by Assumptions \ref{assuniflow} and \ref{assgausslow}, for the Lorenz 63' and the 5 dimensional Lorenz 96' models (see Section \ref{secapplicationlorenz96} for a definition of the Lorenz 96' model).

In the case of the Lorenz 63' model with classical parameter values (\eqref{lorenz63eq1}-\eqref{lorenz63eq3}), for starting point $\vct{u}=(1,2,3)$, an approximation of the leaf set $U(\vct{u})$ is constructed  as follows. We first simulate $\u(t)$ for time $0\le t\le 5$. After this, we sample 100 points from the small neigbourhood $\{\v: \|\v-\u(5)\|_{\infty}\le 10^{-31}\}$, and run them backwards in time until time point $0$. The points we have obtained this way are the small black circles shown on Figure \ref{fig63sm:sub1}.  We repeat this procedure by simulating $\u(t)$ up to time 10, sampling 100 points from the small neigbourhood $\{\v: \|\v-\u(10)\|_{\infty}\le 10^{-62}\}$, and running them backwards in time until time point $0$. These points are shown with grey circles on Figure \ref{fig63sm:sub1}. Finally, the initial point $\u$ is denoted by a big black circle.
As we can see, the grey and black points seem to be part of the same curve, arguably an approximation of the leaf set $U(\vct{u})$. For points $\v$ that are on this curve, $\|\v(t)-\u(t)\|_1$ decreases very quickly in $t$ ( typically at exponential rate). 

As $h\to 0$, and $kh\to t$ for some $t>0$, the sums $D_k^{(1)}(\u,\v)$ and $D_k^{(2)}(\u,\v)$ satisfy that
\[h D_k^{(1)}(\u,\v)\to \int_{s=0}^t \|\v(s)-\u(s)\|_1 ds, \text{ and }h D_k^{(2)}(\u,\v)\to \int_{s=0}^t \|\v(s)-\u(s)\|_2^2 ds.\]

Figures \ref{fig63sm:sub2} and \ref{fig63sm:sub3} plot these integrals up to time $t=5$ started from points $\v$ on our approximation of $U(\vct{u})$, as a function of $\|\v-\u\|_1$ and $\|\v-\u\|_2^2$, respectively.
As we can see, these plots are approximately linear, suggesting that Assumptions \ref{assuniflow} and \ref{assgausslow} are reasonable. This is not a rigorous proof of Assumptions \ref{assuniflow} and \ref{assgausslow} for $\vct{u}=(1,2,3)$, since they concern the supremum for $k\in \N$, and we only look at $\u(t)$ for $0\le t\le 5$. However, by rigorous computations this argument can be extended to imply that the lower bounds of Theorems  \ref{thmlowerboundsmoother} and \ref{thmgausslowerboundsmoother} hold for $\vct{u}=(1,2,3)$ and $0\le t_k\le 5$.

We repeat this same procedure for the 5 dimensional Lorenz 96' model (see Section \ref{secapplicationlorenz96}), started from $\u=(1,2,3,4,5)$. First we simulate $\u(t)$ up to time 5, and 100 points sampled from  $\{\v(5): \|\v(5)-\u(5)\|_{\infty}\le 10^{-32}\}$ are ran backwards until time 0, and then we simulate $\u(t)$ up to time 10, and 100 points sampled from  $\{\v(10): \|\v(10)-\u(10)\|_{\infty}\le 2\cdot 10^{-59}\}$ are ran backwards until time 0. The first 3 coordinates of these are  illustrated by black, and grey points, respectively, on Figure \ref{fig96sm:sub1}. These again seem to be on the same curve (and we obtain similar results if we choose different coordinates), which is an approximation of $U(\vct{u})$. Figures \ref{fig96sm:sub2} and \ref{fig96sm:sub3} illustrate the integrals $ \int_{s=0}^t \|\v(s)-\u(s)\|_1 ds$, and $\int_{s=0}^t \|\v(s)-\u(s)\|_2^2 ds$ up to time $t=5$. As we can see, these are again approximately linear, in good accordance with Assumptions \ref{assuniflow} and \ref{assgausslow}. 

Note that for higher dimensional systems, it is likely that the  stable manifold is no longer a curve but a higher dimensional manifold instead. In such situations, to check our assumptions numerically, for some $T>0$ one can sample $\v(T)$ randomly from a small line segment containing $\u(T)$, and ran it backwards in time to time $0$. The resulting curve set of points $\v(0)$ will be a one dimensional curve (a subset of the stable manifold) that can be chosen as the leaf set $U(\vct{u})$ .

\subsubsection{Assumptions for the filter}
We will now look at Assumption \ref{assuniflowfilter} for the filter, in the case of the Lorenz 63' equations. We construct a possible choice of the anti-leaf sets $\tU(\u,k)$ as follows. 
For $t_k$ fixed, we find the direction of $\v-\u$ where $\frac{\|\v(t_k)-\u(t_k)\|}{\|\v-\u\|}$ is maximal, for $\v-\u$ infinitesimally small (this can be done by computing the Jacobian matrix, and finding its eigenvector corresponding to its maximal eigenvalue). After this, we choose $\tU(\u,k)$ as a small line segment started from $\u$ along this direction. We choose 20 points on this segment (of equal distance between neighbouring points), and run the Lorenz 63' equations up to time $t_k$ started at these points, and evaluate the differences $\int_{s=0}^{t_k} \|\v(s)-\u(s)\|_1 ds$ (approximating $h D_k^{(1)}(\u,\v)$ when $h$ is small). In the case of $t_k=9.2$, for starting point $\u=(1,2,3)$, the value of these integrals is plotted as a function of $\|\v(t_k)-\u(t_k)\|$ on Figure \ref{fig63fi:sub1}. As we can see, this is approximately linear up to a certain distance, and thus \eqref{eqasslowfilter1} holds with $C_{\tU}(\u,k)$ being close to the slope of the linear part. 
We have repeated this experiment again for $t_k\in \{0.2, 0.4, \ldots, 10\}$, and plotted the approximate values of $C_{\tU}(\u,k)$ on Figure \ref{fig63fi:sub2}. In each time point, the constant $\tilde{d}_{\max}(\u,k)$ can be chosen to be greater than $1$. As we can see from this figure, the constant $C_{\tU}(\u,k)$ oscillates and does not seem to tend to infinity as $k$ tends to infinity, in accordance with Assumption \ref{assuniflowfilter}.
 
\begin{figure}
\begin{subfigure}{\linewidth}
\centering
\includegraphics[width=0.5\linewidth]{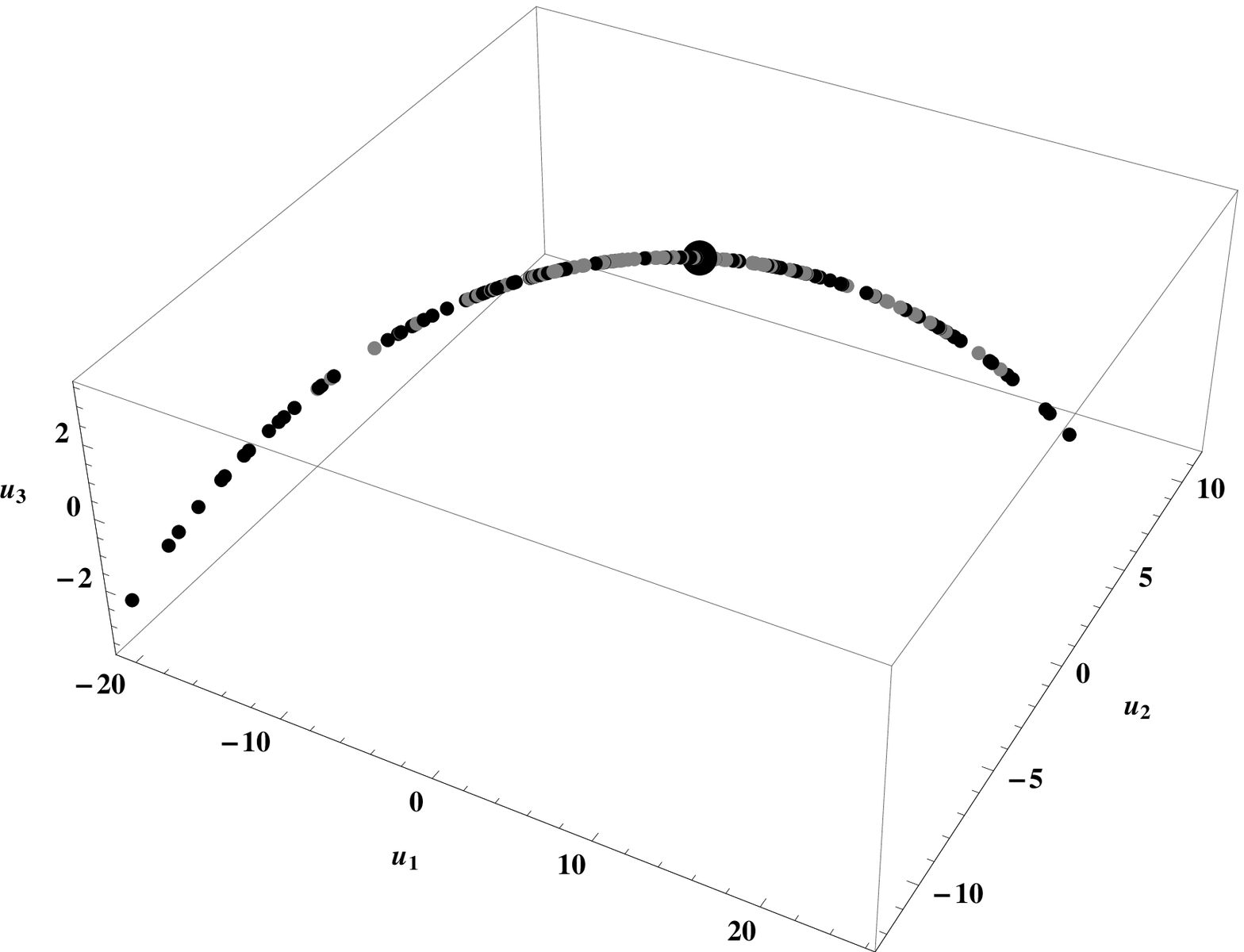}
\caption{}
\label{fig63sm:sub1}
\end{subfigure}%
\\[1ex] 
\begin{subfigure}{.5\linewidth}
\centering
\includegraphics[width=0.8\linewidth]{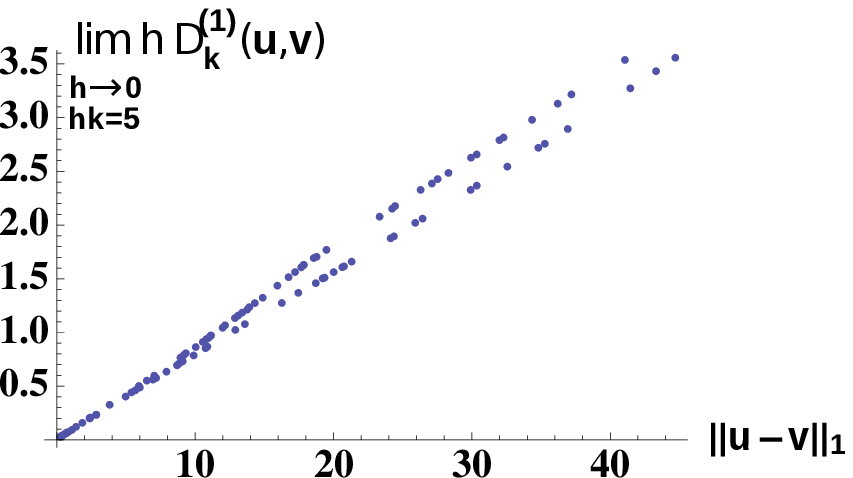}
\caption{}
\label{fig63sm:sub2}
\end{subfigure}
\begin{subfigure}{0.5\linewidth}
\centering
\includegraphics[width=0.8\linewidth]{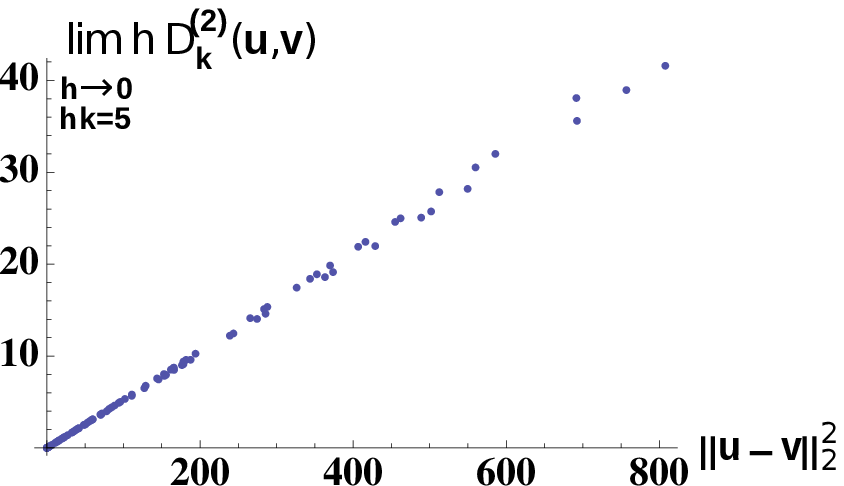}
\caption{}
\label{fig63sm:sub3}
\end{subfigure}
\caption{Illustration of the leaf set $U(\vct{u})$ and its properties for the Lorenz 63' model }
\begin{subfigure}{\linewidth}
\centering
\includegraphics[width=0.5\linewidth]{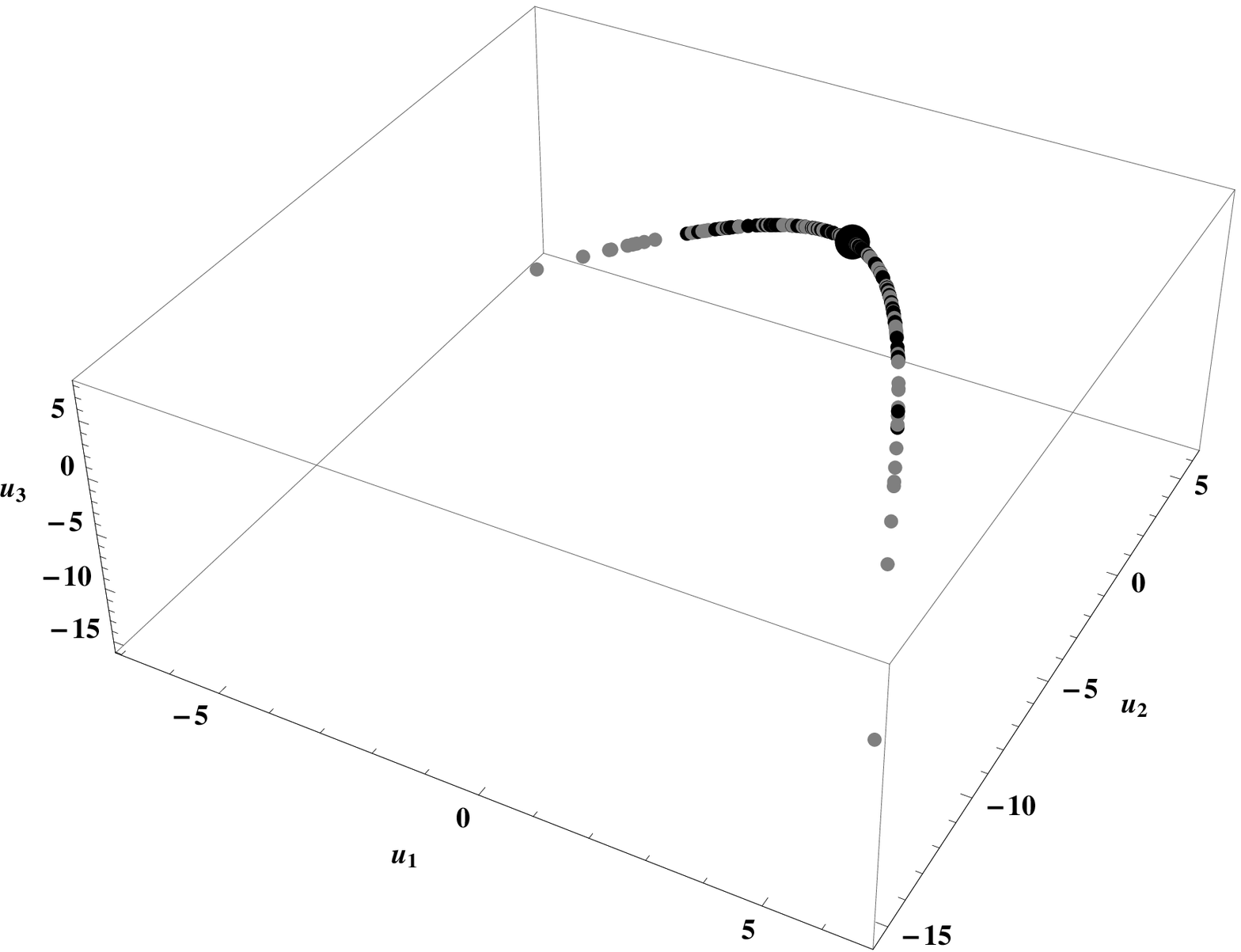}
\caption{}
\label{fig96sm:sub1}
\end{subfigure}
\\[1ex]
\begin{subfigure}{0.5\linewidth}
\centering
\includegraphics[width=0.8\linewidth]{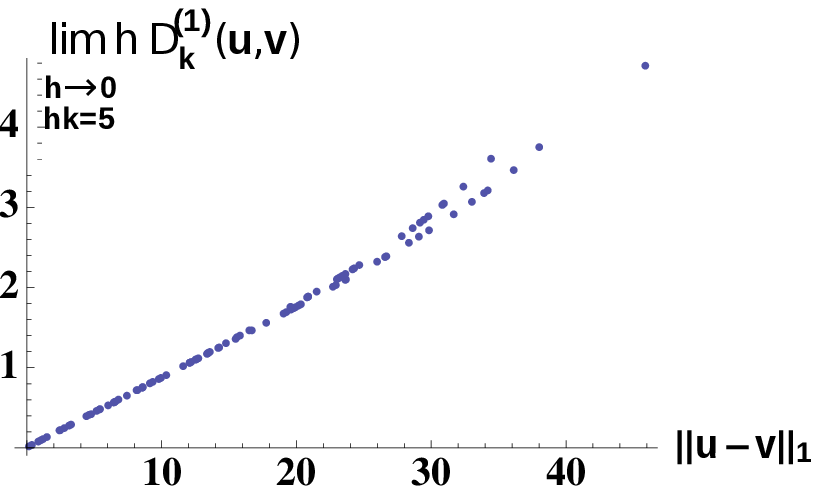}
\caption{}
\label{fig96sm:sub2}
\end{subfigure}
\begin{subfigure}{0.5\linewidth}
\centering
\includegraphics[width=0.8\linewidth]{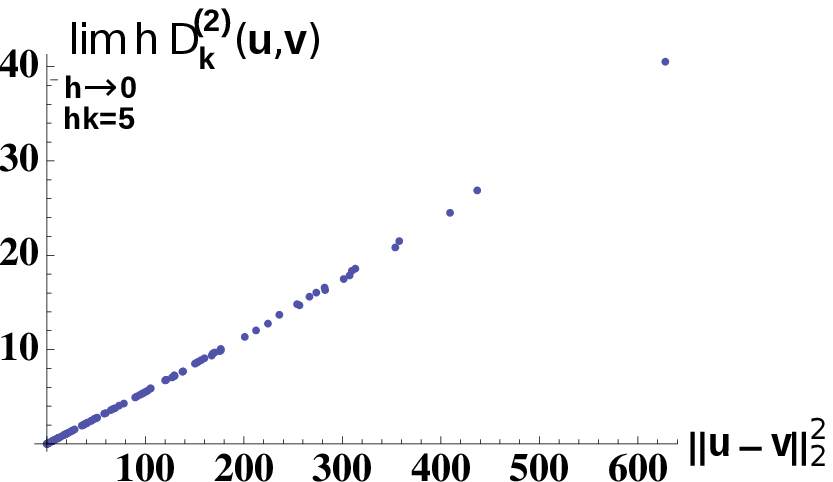}
\caption{}
\label{fig96sm:sub3}
\end{subfigure}
\caption{Illustration of  the leaf set $U(\vct{u})$ and its properties for the Lorenz 96' model }
\label{fig:test}
\end{figure}

\begin{figure}
\begin{subfigure}{.5\linewidth}
\centering
\includegraphics[width=\linewidth]{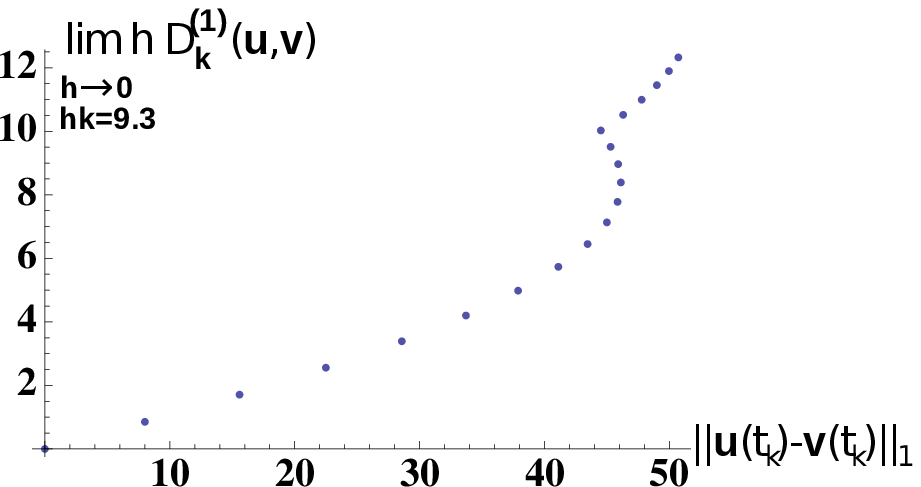}
\caption{}
\label{fig63fi:sub1}
\end{subfigure}
\begin{subfigure}{0.5\linewidth}
\centering
\includegraphics[scale=1,bb=0 0 246 160, width=\linewidth]{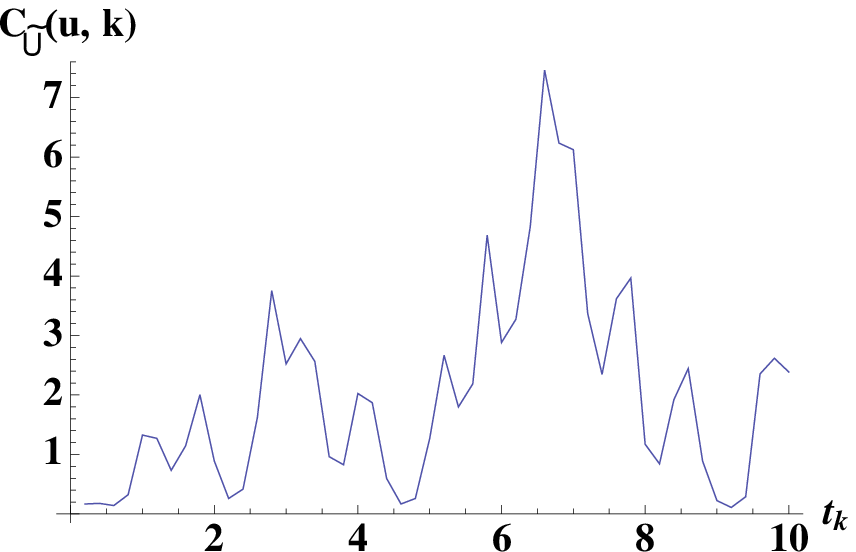}
\caption{}
\label{fig63fi:sub2}
\end{subfigure}\caption{Illustration of the properties of the anti-leaf sets $\tU(\vct{u},k)$ for the Lorenz 63' model}
\label{fig:test}
\end{figure}

\section{Upper bounds}\label{secupperbounds}
In this section we establish upper bounds for the smoother and the filter. In the case of bounded observation errors, we will give some conditions that guarantee that  the diameter of the support of the smoother (or the filter) are upper bounded by a constant times the size of the noise. In the case of unbounded observation errors, we show that under the same assumptions, there is an estimator based on the observations whose mean square error from the true position is upper bounded by a constant times the variance of the noise. 
We show that the assumptions required by our results can be deduced from the fact that a certain  system of polynomial equations has a unique solution. In Sections \ref{secapplicationlorenz63} and \ref{secapplicationlorenz96}, we apply our results to the Lorenz 63' and Lorenz 96' models. In Section \ref{secapplicationrandomcoeff}, we verify our assumptions for some 3 and 4 dimensional systems with random coefficients, when only the first coordinate is observed.
\subsection{Results}\label{secresultsupperbounds}
Let us define the observed part of the one parameter solution semigroup as
\begin{equation}\label{eqPhidef}
\Phi_t(\u):=\mtx{H} \Psi_t(\u) \text{ for } \u\in \BR, t\in \R_+.
\end{equation}

For our upper bounds we make the following assumption on the dynamics, the prior, and the initial point $\u$.
\begin{ass}\label{ass1} Suppose that there is an index $k\in \N$ and a positive constant $c(\u,k)$ such that for any $\v\in \BR$,
\[\max_{0\le i\le k} \| \Phi_{t_i}(\u)- \Phi_{t_i}(\v) \| \ge c(\u,k) \|\v- \u\|.\]
\end{ass}

Assumption \ref{ass1} quantifies how much the differences $\|\Phi_{t_i}(\v) - \Phi_{t_i}(\u)\|$ grow as we move away from $\u$. This assumption seems to be rather strong at first, since they involve ``global'' assumptions about $\Phi_{t_i}$, which can behave rather chaotically. 
However, as we shall see in Proposition \ref{Propcheckassumptions}, it is possible to deduce it from ``local'' assumptions about the derivatives of $\Phi$ at time $0$. These ``local'' assumptions in turn can be easily checked for the Lorenz 63' and Lorenz 96' models when the partial observations are chosen suitably (see Sections \ref{secapplicationlorenz63} and \ref{secapplicationlorenz96}). We believe that these assumptions hold for many observations scenarios in a wide range of dynamical systems such as Garelkin spectral truncations of the Navier--Stokes equations, and various discretisations of the shallow-water equations. Since these assumptions essentially only require that the observed components of the system given sufficiently many observations uniquely determine the initial position, we believe that they are more generally applicable than earlier consistency results for the 3D-Var shown in 
\cite{AlonsoStuartLongtime} and \cite{law2014controlling}.

Our assumptions on the derivatives are stated as follows.
\begin{ass}\label{assder}
Suppose that $\|\u\|<R$, and there is an index $j\in \N$ such that the system of equations in $\v$ defined as
\begin{equation}\label{eqphitideruv} 
\H\D^i \u= \H \D^i \v\text{ for every }0\le i\le j\end{equation}
has a unique solution $\v:=\u$ in $\BR$, and
\begin{equation}\label{eqspanderu}
\mr{span}\l\{\grad \l(\left(\H \D^i \u\right)_k\r): 0\le i\le j, 1\le k\le d_o\r\}=\R^{d},
\end{equation}
where $\grad$ denotes the gradient of the function in $\u$, and $\left( \right)_k$ denotes the $k$th coordinate.
\end{ass}

One can see that \eqref{eqspanderu} is equivalent to 
\begin{equation}\label{eqlambdaminder}
\lambda_{\min}\left[\sum_{i=0}^{j}\sum_{k=1}^{d_o}\grad \l(\left(\H \D^i \u\right)_k\r) \cdot \left(\grad \l(\left(\H \D^i \u\right)_k\r) \right)'\right]>0.
\end{equation}

\begin{prop}[Assumptions on derivatives imply assumptions for upper bounds]\label{Propcheckassumptions}
Suppose that Assumption \ref{assder} holds. Then for sufficiently small $h$, Assumption \ref{ass1} holds for every $k\ge j$.
\end{prop}
The proof of this proposition based on Taylor's expansion. It is included in Section \ref{secassderproof} of the Appendix.

Now we are ready to state our upper bounds. In our first result, we will assume that the observation errors satisfy that $\|\vct{Z}_i\|\le \epsilon$ almost surely. Given the observations $\vct{Y}_0, \vct{Y}_1,\ldots, \vct{Y}_k$, the support of the smoothing distribution for $\epsilon$-bounded observation errors ($\|Z_i\|\le \epsilon$ almost surely for every $i\in \N$) is contained in 
\begin{equation}
\Lambda_k^{(\epsilon)}:=\left\{\v\in \BR: \max_{0\le i\le k}\|\vct{Y}_i-\Phi_{t_i}(\v)\|\le \epsilon\right\}.
\end{equation}
Alternatively, we can define the $(k,\epsilon)$ neighbourhood of the true initial point $\u$ as
\begin{equation}
\Omega_k^{(\epsilon)}:=\left\{\v\in \BR: \max_{0\le i\le k}\|\Phi_{t_i}(\v)-\Phi_{t_i}(\u)\|\le \epsilon\right\}.
\end{equation}
By the triangle inequality, we have $\Lambda_k^{(\epsilon)}\subset \Omega_k^{(2\epsilon)}$.

\begin{theorem}[Upper bound for bounded observation errors]\label{thmupperbounded}
Under Assumption \ref{ass1}, for any $\epsilon>0$, we have
\begin{equation}\label{suppOmegaLambdaeq}\diam \left(\Omega_k^{(\epsilon)}\right)\le c(\u,k) \epsilon, \quad \text{ and thus }\quad \diam \left(\Lambda_k^{(\epsilon)}\right)\le 2c(\u,k) \epsilon.\end{equation}
Thus for $\epsilon$-bounded observation errors ($\|Z_i\|\le \epsilon$ almost surely for every $i\in \N$) the support of the smoother is bounded as
\begin{equation}\label{suppsmupperbndeq}
\mathrm{diam}\,\supp\musm(\cdot|\vct{Y}_0,\ldots, \vct{Y}_{k})\le 2c(\u,k) \epsilon,
\end{equation}
and the support of the filter is bounded as
\begin{equation}\label{suppfiupperbndeq}
\mathrm{diam}\,\supp\mufi(\cdot|\vct{Y}_0,\ldots, \vct{Y}_{k})\le 2c(\u,k)e^{Gt_k}  \epsilon,
\end{equation}
with the constant $G$ defined as in \eqref{Gdef}.
\end{theorem}
\begin{proof}
\eqref{suppOmegaLambdaeq} directly follows from Assumption \ref{ass1}. \eqref{suppsmupperbndeq} follows from the fact that the support of the smoother is included in the set $\Lambda_k^{(\epsilon)}$. Finally, \eqref{suppfiupperbndeq} is implied by \eqref{eqpathdistancebound}, and the fact that
the support of the filter is included in the set $\Psi_{t_k}(\Lambda_k^{(\epsilon)}):=\{\Psi_{t_k}(\v): \v\in \Lambda_k^{(\epsilon)}\}$.
\end{proof}

The following result concerns the case of unbounded observation errors.

\begin{theorem}[Upper bound for unbounded observation errors]\label{thmupperunbounded}
Suppose that Assumption \ref{ass1} holds, and that
\[\sigma_Z^2:=\E(\|\vct{Z}_i\|^2)<\infty.\]
Let
\[E_{\max}(\v|\vct{Y}_0,\ldots,\vct{Y}_k):=\max_{0\le i\le k} \|\Phi_{t_i}(\v)-Y_i\|,\]
and
\begin{equation}\label{eqUinftydef}
\u_{\min}:=\argmin_{\v\in \BR} E_{\max}(\v|\vct{Y}_0,\ldots,\vct{Y}_k).
\end{equation}
If there are multiple minima, than we can define this function as any one of them. Then the estimator $\u_{\min}$ of the initial position $\u$ satisfies that
\begin{equation}\label{uminerreqsm}\E\left(\left.\|\u_{\min}-\u\|^2\right|\u\right)\le D(\u, k) \cdot \sigma_Z^2,
\end{equation}
for some constant $D(\u,k)<\infty$.

Moreover, the push-forward map of $\u_{\min}$, $\Psi_{t_k}(\u_{\min})$, is an estimator of the current position $\u(t_k)$, satisfying that
\begin{equation}\label{uminerreqfi}\E\left(\left.\|\Psi_{t_k}(\u_{\min})-\u(t_k)\|^2\right|\u\right)\le D(\u, k) e^{2G t_k}\cdot \sigma_Z^2,
\end{equation}
with the constant $G$ defined as in \eqref{Gdef}.
\end{theorem}
\begin{proof}
Let $Z_{\max}^{(k)}:=\max_{0\le i\le k}\frac{\|\vct{Z}_i\|}{\sigma_Z}$, then 
\[
E_{\max}(\u_{\min}|\vct{Y}_0,\ldots,\vct{Y}_k)\le E_{\max}(\u|\vct{Y}_0,\ldots,\vct{Y}_k)=Z_{\max}^{(k)}\sigma_Z,
\]
and thus by the triangle inequality, we have
\[\max_{0\le i\le k}\|\Phi_{t_i}(\u_{\min})-\Phi_{t_i}(\u)\|\le 2Z_{\max}^{(k)}\sigma_Z.\]
Therefore $\u_{\min}\in \Omega_k^{2 Z_{\max}^{(k)} \sigma_Z}$, and \eqref{uminerreqsm} follows by Theorem \ref{thmupperbounded}, with \[D(\u,k):=4 \E\left((Z_{\max}^{(k)})^2\right)\cdot c(\u,k)^2\le 4(k+1)c(\u,k)^2.\] Finally, \eqref{uminerreqfi} follows by \eqref{eqpathdistancebound}.
\end{proof}

\subsection{Application to the Lorenz '63 model}\label{secapplicationlorenz63}
As shown on page 16 of \cite{AlonsoStuartLongtime}, the Lorenz equations \eqref{lorenz63eq1}-\eqref{lorenz63eq3} can be transformed to the form of \eqref{diffeqgeneralform} by a linear change of coordinates. In this case, the coefficients of the equation are given by
\begin{align*}
\mtx{A}=\left[\begin{matrix}a& -a& 0\\a& 1& 0\\ 0& 0& b \end{matrix}\right],\quad \mtx{B}(u,\tilde{u})=\left[\begin{matrix}0\\(u_1\tilde{u}_3+u_3\tilde{u}_1)/2\\ -(u_1\tilde{u}_2+u_2\tilde{u}_1)/2\end{matrix}\right], \quad \mtx{f}=\left[\begin{matrix}0\\0\\ -b/r+a\end{matrix}\right].
\end{align*}

We choose the observation operator as $\mtx{H}:=\left[\begin{matrix}1& 0& 0\\0& 0& 0\\ 0& 0& 0 \end{matrix}\right]$. This corresponds to observing the first coordinate $u_1$ of the process.

The following proposition shows that our theory applies here.
\begin{prop}\label{propupperLorenz63}
For $j\ge 2$, for Lebesgue almost every initial point $\u\in \BR$, Assumption \ref{assder} holds for the process described above. 

As a consequence, for $\epsilon$-bounded observation errors ($\|Z_i\|\le \epsilon$ almost surely for every $i\in \N$), for almost every initial point $\u\in \BR$, for sufficiently small $h$, the diameter of the support of the smoother $\musm(\v|\vct{Y}_0,\ldots,\vct{Y}_k)$ and the filter $\mufi(\v|\vct{Y}_0,\ldots,\vct{Y}_k)$ can be bounded by $C_{\mathrm{sm}}(\u,k) \epsilon$ and $C_{\mathrm{fi}}(\u,k) \epsilon$, respectively, for some finite constants $C_{\mathrm{sm}}(\u,k)$ and $C_{\mathrm{fi}}(\u,k)$ which do not depend on $\epsilon$.

Moreover, for unbounded observation satisfying that $\sigma_Z^2:=\E(\|\vct{Z}_i\|^2)<\infty$, for almost every initial point $\u\in \BR$, for sufficiently small $h$, there are some estimators based on the observations, $U_{\mathrm{sm}}(\vct{Y}_0,\ldots,\vct{Y}_k)$ and $U_{\mathrm{fi}}(\vct{Y}_0,\ldots,\vct{Y}_k)$, such that the mean square errors 
$\E \left[(U_{\mathrm{sm}}(\vct{Y}_0,\ldots,\vct{Y}_k)-\u)^2\right]$ and $\E \left[(U_{\mathrm{fi}}(\vct{Y}_0,\ldots,\vct{Y}_k)-\u(t_k))^2\right]$ of the initial and current positions are bounded by
$D_{\mathrm{sm}}(\u,k) \sigma_Z^2$ and $D_{\mathrm{fi}}(\u,k) \sigma_Z^2$, respectively, for some finite constants $D_{\mathrm{sm}}(\u,k)$ and $D_{\mathrm{fi}}(\u,k)$ which do not depend on $\sigma_Z$.
\end{prop}

\begin{proof}[Proof of Proposition \ref{propupperLorenz63}]
Due to the definition of the observation operator, we have $\H \D^i(\v) = \D^i v_1$. Now based on the equations \eqref{diffeqgeneralform}, we have
\begin{align}
\label{eqv11}\D^0 v_1&=v_1,\\
\label{eqv12}\D^1 v_1&=-av_1+av_2,\\
\label{eqv13}\D^2 v_1&=-a \D v_1+a\D v_2\\
&=-a(-av_1+av_2)+a(-av_1-v_2 -v_1 v_3)\\
&=-2a^2 v_1 - (a^2+a)v_2-av_1 v_3.
\end{align}
Based on this, we can express $v_1$, $v_2$ and $v_3$ as a function of $\frac{d^0}{d t^0}v_1(0)$, $\frac{d}{d t}v_1(0)$ and $\frac{d^2}{d t^2}v_1(0)$ as
\begin{align}
v_1&=\frac{d^0}{d t^0}v_1(0),\\
v_2&=\frac{1}{a}\frac{d}{d t}v_1(0)+ \frac{d^0}{d t^0}v_1(0),\\
v_3&=\frac{\frac{d^2}{d t^2}v_1(0) -(a+1)\frac{d}{d t}v_1(0) +(a^2-a) \frac{d^0}{d t^0}v_1(0) }{a \frac{d^0}{d t^0}v_1(0)}  \text{ if } v_1\ne 0.
\end{align}
These explicit expressions imply that the condition \eqref{eqphitideruv} holds for almost every $\u\in \BR$. Condition \eqref{eqspanderu} is satisfied because of the upper triangular form of the equations \eqref{eqv11}-\eqref{eqv13}. The claims on the smoother and the filter now directly follow from Theorems \ref{thmupperbounded} and \ref{thmupperunbounded}.
\end{proof}

\subsection{Application to the Lorenz '96 model}\label{secapplicationlorenz96}
The Lorenz '96 model is a  $d$ dimensional chaotic dynamical system which was introduced in \cite{lorenz1996predictability}. As shown on page 16 of \cite{AlonsoStuartLongtime}, it can be written in the framework of \eqref{diffeqgeneralform} as
\begin{align*}
\mtx{A}=\mtx{I}_{d\times d},  \mtx{B}(u,\tilde{u})=-\frac{1}{2}\left[\begin{matrix}\vdots\\ \tilde{u}_{i-1}u_{i+1}+u_{i-1}\tilde{u}_{i+1} -\tilde{u}_{i-2}u_{i-1}-u_{i-2}\tilde{u}_{i-1}  \\ \vdots\end{matrix}\right],\mtx{f}=\left[\begin{matrix}8\\ \vdots \\ 8\end{matrix}\right],
\end{align*}
where the indices of $\u$ in the expression of $\mtx{B}$ are understood modulo $d$. The observation matrix $\mtx{H}$ is defined as $\mtx{H}_{1,1}=\mtx{H}_{2,2}=\mtx{H}_{3,3}=1$, and $0$ in every other element. This means that we observe the first 3 coordinates $u_1$, $u_2$, and $u_3$.
The following proposition shows that our theory is applicable to this situation.

\begin{prop}\label{propupperLorenz96}
For $j\ge d-3$, for Lebesgue almost every initial point $\u\in \BR$, Assumption \ref{assder} holds for the process described above. As a consequence, the same results hold for the smoother and the filter as in Proposition \ref{propupperLorenz63}.
\end{prop}

\begin{proof}[Proof of Proposition \ref{propupperLorenz96}]
Because of the definition of the observation operator, we have $\H \D^i \u = \left(\D^i u_1,\D^i u_2,\D^i u_3  \right)$. Based on the equations \eqref{diffeqgeneralform}, we have 
\[\D u_3=8-u_3+u_2 u_4 -u_1u_2,\]
and thus we are able to write
\[u_4=\left(\D u_3-8+u_3+u_1v_2\right)/v_2.\]
Due to the specific multi-diagonal structure of $\mtx{B}(\u,\u)$ (the $i$th column only depends on the $i-2,\ldots, i+1$th terms), by repeatedly expressing the derivatives $\frac{d}{d t}u_j$ for $j\ge 4$ one-by-one, we can obtain a similar deterministic expression for $u_5,\ldots, u_d$ just in terms of the derivatives
$ \left(\D^i u_1,\D^i u_2,\D^i u_3  \right)$, for $0\le i\le d-3$. The equations are valid almost surely, for every $\u$ such that $u_i\ne 0$ for every $1\le i\le d$. These explicit expressions imply that condition \eqref{eqphitideruv} holds for almost every $\u\in \BR$.
Now we are going to verify condition \eqref{eqspanderu}.  Suppose that $\u$ satisfies that  $u_i\ne 0$ for every $1\le i\le d$. Let $e_i$ be a $d$ dimensional unit vector with $1$ in coordinate $i$ and 0 elsewhere. Then $\grad u_1=e_1,\ldots, \grad u_3=e_3$. From the definition of the model, we have
\[\D u_i=-u_i-u_{i-1}u_{i+1}-u_{i-1}u_{i-2}+f_i,\]
where the indices are meant modulo $d$. This implies that 
\[\grad \D u_i=-e_{i+1}u_{i-1}-e_i-e_{i-1}(u_{i+1}+u_{i-2})-e_{i-2}u_{i-1},\]
and by our assumption on $\u$, we have
\[\mathrm{span}\left(\grad u_1,\ldots, \grad u_3, \grad \D u_1, \ldots, \grad \D u_2, \grad \D u_3\right)=\mathrm{span}(e_{d-1},e_d,e_1,\ldots, e_4).\]
By adding the higher order derivatives one by one, we obtain \eqref{eqspanderu} for every $\u$ satisfying our assumption that $u_i\ne 0$ for every $1\le i\le d$. The consequences about the smoother and the filter follow the same way as in the proof of Proposition \ref{propupperLorenz63}.
\end{proof}

\subsection{Application to systems with random coefficients}\label{secapplicationrandomcoeff}
If a dynamical system of the form \eqref{diffeqgeneralform} satisfies Assumption \ref{assder}, then our upper bounds are valid. In the previous two examples, we have shown that for two particular systems, under suitably chosen partial observations, this assumption is satisfies. In order to check how restrictive is this assumption, we have done the following experiment. We have chosen the elements of $\mtx{A}, \mtx{B}, \vct{f}, \u$ randomly, independently of each other, uniformly on the set $\{1,2,\ldots,10\}$, and checked Assumption \ref{assder} by a Mathematica code, which is available on request. We have done 100 random trials for 3 dimensional systems, with the first 3 derivatives (thus $j=3$), with only the first coordinate observed, and found that all of them satisfy Assumption \ref{assder}. We have repeated this experiment with 4 dimensional systems (with $j=4$), and obtained the same result. 

These results are consistent with the intuition that if all of the coordinates of the system interact with each other, then it should be possible to interfere the position of the system by observing only one coordinate of it with sufficiently high precision. The simulation results suggest that the set of coefficients and initial positions $\mtx{A}, \mtx{B}, \vct{f}, \u$ where Assumption \ref{assder} does not hold probably has Lebesgue measure 0 (however, proving this is beyond the scope of this paper).

\subsubsection*{Acknowledgements}
All authors were supported by the Singapore Ministry of Education AcRF tier 2 grant  R-155-000-161-112. The work of the third author has been partially supported by the EPSRC grant  EP/N023781/1. We thank the anonymous referees for their insightful comments.

\appendix

\section{Appendix}
\subsection{Proof of the existence of anti-leaf sets for the geometric model}\label{secgeoantileafproof}
In this section, we are going to prove Theorem \ref{thmgeoantileafsets}. The proof is based on the strong expansion property of $f$ ($|f'(x)|>\sqrt{2}$ for every $x\in [-1/2,1/2]\setminus \{0\}$).  Using this property, we are going to define closed intervals on $[-1/2,1/2]$ satisfying certain requirements, and then define the sets $\tU(\u,k)$ based on these intervals.

For $j\in \N$, let $f^{(j)}$ denote the composition of $f$ with itself $j$ times (with $f^{(0)}(x):=x$), and denote by $D^{(j)}_f$ the domain of $f^{(j)}(x)$. For any set $W\subset [-1/2,1/2]\setminus \{0\}$, we let $f(W):=\{f(x): x\in W\}$, and similarly, for any set $W\subset [-1/2,1/2]$, we let 
$f^{(-1)}(W):=\{x: x\in [-1/2,1/2]\setminus \{0\}, f(x)\in W\}$. 
Note that due to the particular structure of $f$, for any closed interval $I\subset [-1/2,1/2]$, $f^{(-1)}(I)$ consists of one or two closed intervals. Let $f^{(-j)}$ denote the composition of $f^{(-1)}$ with itself $j$ times. Then the domains $D^{(j)}_f$ can be expressed as
\begin{equation}
D^{(j)}_f:=[-1/2,1/2]\setminus \left(\cup_{i=0}^{j-1} f^{(-i)}(\{0\})\right).
\end{equation}

The next lemma defines the intervals $I^{(j)}_x$ and proves that they satisfy certain properties.
\begin{lem}[Definition of the intervals $I^{(j)}_x$]\label{Ijxdeflemma}
For $x\in [-1/2,1/2]$, let $d(x):=\min(|x-1/2|,|x|,|x+1/2|)$ be the distance between $x$ and the set $\{-1/2,0,1/2\}$. For any $j\in \N$, $x\in D^{(j)}_f$, let
\begin{equation}\label{deltakdef}
\delta^{(j)}_x:=\min_{0\le i\le j} \left(2^{(j-i)/2}\cdot d(f^{(i)}(x))\right).
\end{equation}
For $x\in D^{(j)}_f$, $0\le i\le j$ we define a sequence of intervals $I^{(j,i)}_x$ as follows. First, let
$I^{(j,j)}_x:=\left[f^{(j)}(x)-\frac{\delta^{(j)}_x}{2},f^{(j)}(x)+\frac{\delta^{(j)}_x}{2}\right]$. The rest of the intervals are defined iteratively, given $I^{(j,i)}_x$ for $1\le i\le j$, we define $I^{(j,i-1)}_x$ as the closed interval in the set $f^{(-1)}(I^{(j,i)}_x)$ containing $f^{(i-1)}(x)$. Finally, let $I^{(j)}_x:=I^{(j,0)}_x$.

Then for any $0\le i\le j$, the sets $f^{(i)}(I_x^{(j)})$ are closed intervals containing $f^{(i)}(x)$ that do not contain 0, and satisfy that
\begin{equation}\label{Ijxpropeq}
\inf \left\{|y|: y\in f^{(i)}(I^{(j)}_x)\right\}\ge \frac{\left|f^{(i)}(x)\right|}{2}.
\end{equation}
\end{lem}
\begin{proof}
Since $|f'(x)|\ge \sqrt{2}$ for every $x\in D^{(1)}_f$, it follows that the length of the interval $I^{(j,i-1)}_x$ is shorter than that the length of the interval $I^{(j,i)}_x$ by at least a factor of $\sqrt{2}$ for every $1\le i\le j$. The stated properties of $I_x^{(j)}$ are now implied by the definition of $\delta^{(j)}_x$.
\end{proof}

Now we are going to define the sets $\tU(\u,k)$ for every $k\in \N, \u\in \Lambdageo$. Let $\TT(\u,k)$ denote the number of time points $t\in (0, t_k]$ such that $\u(t)\in S$ (i.e.\ the number of turns taken by the geometric model started from $\u$ until time $t_k$). Let 
\[U^*(\u,k):=\{\v\in S: v_1\in I_{O_1(\u)}^{(\TT(u,k))}, v_2=O_2(\u)\},\]
that is a small line segment on $S$ in direction parallel to the axis $u_1$ containing the point $O(\u)$.
For any $\v,\w\in \R^3$, we denote by $[\v,\w]:=\{a \v+(1-a)\w: a\in [0,1]\}$ the line segment between $\v$ and $\w$. We define the \emph{anti-leaf sets} by propagating this set forward by $\tau_O(\u)$ time, and imposing an additional condition as
\begin{align}\label{geoantileafsetdefeq}
\tU(\u,k):=&\bigg\{ \w(t_k+\tau_O(\u)): \w\in U^*(\u,k) \text{ such that for every }\vct{z}\in [O(\u),\w],\\
\nonumber&\|\vct{z}(t_k+\tau_O(\u))-\u(t_k)\|_{\infty}\le 0.05\bigg\}.
\end{align}
This additional condition will guarantee that if $\u(t_k)$ is sufficiently near $S$, then $\Psi^{\mathrm{geo}}(\U^*(\u,k))\subset W^{S}_{0.1}$, which will be useful in the following argument.

The next two lemmas bound the difference $\|\v(t)-\u(t)\|_{\infty}$ for two points $\u,\v\in S^*$.

\begin{lem}[Maximal distance between two paths by the first return]\label{maxdistancefirstretlemma}
Let $\u$ and $\v$ be two points on $S^*$ satisfying that $|u_1-v_1|\le \frac{|u_1|}{2}$. Then for $0\le t\le \tau(\u)$ (the time it takes to return to $S$ from $\u$), we have 
\begin{equation}\label{maxdisteq}
\|\v(t)-\u(t)\|_{\infty}\le C_{1}\left(\frac{|u_1-v_1|}{|u_1|}+|u_2-v_2|\right),
\end{equation}

for a constant $C_1$ only depending on the parameters of the model.
\end{lem}
\begin{proof}[Proof of Lemma \ref{maxdistancefirstretlemma}]
For the linear part of the dynamics, we have
\[\Psi^{\mathrm{lin}}_t(\u)-\Psi^{\mathrm{lin}}_t(\v)= \left((u_1-v_1) e^{\lambda_1 t}, (u_2-v_2) e^{-\lambda_2 t}, 0 \right)\]
for $0\le t\le \min[\tau_{\Sigma}(u_1),\tau_{\Sigma}(v_1)]$.
Thus until the time the first one of the paths reaches $\Sigma$, their distance is bounded as
\begin{equation}\label{distboundpart1eq}
\sup_{0\le t\le \min[\tau_{\Sigma}(\u),\tau_{\Sigma}(\v)]}\|\u(t)-\v(t)\|_{\infty}\le \frac{|u_1-v_1|}{|u_1|}+|u_2-v_2|.\end{equation}
 The difference between the time they take from $S^*$ to $\Sigma$ can be bounded as
\begin{equation}\label{timedelayeq}|\tau_{\Sigma}(u_1)-\tau_{\Sigma}(v_1)|=\frac{1}{\lambda_1} |\log(v_1/u_1)|\le \frac{2}{\lambda_1} \frac{|u_1-v_1|}{|u_1|}.
\end{equation}
The two paths started at $\u$ and $\v$ will reach $\Sigma$ at $L(\u)$ and $L(\v)$ (see \eqref{Ldefeq}), and the distance of these two points can be bounded as
\begin{align}\nonumber|L_1(\u)-L_1(\v)|&=\left| u_2|u_1|^{\beta}-v_2|v_1|^{\beta}\right|\le  |u_2-v_2| |u_1|^{\beta}+ |v_2| \left||v_1|^{\beta}-|u_1|^{\beta}\right|
\\
\label{L1bound}&\le \frac{1}{2}|u_2-v_2| + \frac{2}{3} |u_1-v_1|,\\
\label{L2bound}|L_2(\u)-L_2(\v)|&=\left| |u_1|^{\alpha}-|u_2|^{\alpha}\right|\le \frac{2\alpha |u_1-v_1|}{|u_1|}.
\end{align}
For the rotation part of the dynamics, by \eqref{SigmaptoSeveq} and \eqref{SigmamtoSeveq}, 
we have for any $\w,\vct{z}\in \Sigma_+$ or $\w,\vct{z}\in \Sigma_-$, $0\le s\le \frac{3\pi}{2}$, 
\begin{align*}\|\Psi_s^{\mathrm{rot}}(\w)-\Psi_s^{\mathrm{rot}}(\vct{z})\|_{\infty}&\le \max(|w_2-z_2|,\theta |w_3-z_3|)\\
&\le |w_2-z_2|+2 |w_3-z_3|,
\end{align*}
so the distance between these paths can not grow by more than by a factor of 2 until they reach $S$. Thus  two paths started at points $L(\u)$ and $L(\v)$ on $\Sigma$ will reach $S$ at the same time, and their distance during this time is bounded as 
\begin{equation}\label{distboundpart2eq}\max_{0\le t\le (3/2)\pi}\|\Psi_t^{\mathrm{geo}}(L(\u))-\Psi_t^{\mathrm{geo}}(L(\v))\|_{\infty}\le 2\left(\frac{1}{2}|u_2-v_2| + \frac{2}{3} |u_1-v_1|+\frac{2\alpha |u_1-v_1|}{|u_1|}\right).\end{equation}
However, the paths started at $\u$ and $\v$  reach $\Sigma$ at different time points, so we still need to account for the time delay. From \eqref{vmaxgeodefeq} we know that the speed of the dynamics is bounded by $\vmaxgeo$, so by equations \eqref{distboundpart1eq}, \eqref{distboundpart2eq}, \eqref{timedelayeq} ands the triangular inequality, the maximal distance between the paths can be bounded as
\begin{align*}&\sup_{0\le t\le \tau(\u)}\|\v(t)-\u(t)\|_{\infty}\le 2\left(\frac{1}{2}|u_2-v_2| + \frac{2}{3} |u_1-v_1|+\frac{2\alpha |u_1-v_1|}{|u_1|}\right)\\
&+|\tau_{\Sigma}(u_1)-\tau_{\Sigma}(v_1)| \vmaxgeo\le 
|u_2-v_2| + \frac{4}{3} |u_1-v_1|+\frac{4\alpha |u_1-v_1|}{|u_1|}+
\frac{2\vmaxgeo}{\lambda_1} \frac{|u_1-v_1|}{|u_1|},
\end{align*}
and the stated result follows with $C_1:=\frac{2}{3}+4\alpha+\frac{2\vmaxgeo}{\lambda_1}$.
\end{proof}

\begin{lem}[Bounding the maximum distance between two paths until their $l$th return]\label{lemmamaxdistanceuntillthreturn}
Let $l\in \N$, and $\u\in S^*$ be such that $\u(t)$ crosses $S$ at least $l+1$ times for $t>0$, and $\v\in S^*$ be such that $u_2=v_2$, and $v_1\in I^{(l)}_{u_1}$ (defined according to Lemma \ref{Ijxdeflemma}). 
Let $T_1(\u), T_2(\u), \ldots$ be the subsequent return times of $\u(t)$ to $S$ (and denote $T_0(\u):=0$).  Then for any $0\le j \le l$, $t\in [T_j(\u), T_{j+1}(\u)]$, we have
\begin{equation}
\|\u(t)-\v(t)\|_{\infty}\le C^{\mathrm{ret}} |f^{(l)}(u_1)-f^{(l)}(v_1)| \sum_{i=0}^{j} \frac{2^{-(l-i)/2}}{|f^{(i)}(u_1)|},
\end{equation}
for some constant $C^{\mathrm{ret}}<\infty$ only depending on the parameters of the model.
\end{lem}
\begin{proof}[Proof of Lemma \ref{lemmamaxdistanceuntillthreturn}]
Based on the definition of $I_{u_1}^{(l)}$, it follows that $\v(t)$ also crosses $S$ at least $l$ times for $t>0$. For $0\le i\le l$, let $\Delta_1(i):=|u_1(T_i(\u))-v_1(T_i(\v))|=|f^{(i)}(u_1)-f^{(i)}(v_1)|$ and  $\Delta_2(i):=|u_2(T_i(\u))-v_2(T_i(\v))|$ be the differences between the return points on the plane. Since the coordinate $u_2$ evolves in a linear fashion during the rotation part of the dynamics, from \eqref{L1bound} we have that
\begin{equation}\label{delta2bound1}
\Delta_2(i+1)\le \frac{1}{2} \Delta_2(i) + \frac{2}{3} \Delta_1(i) \text{ for }0\le i<l.
\end{equation}
By the definition of  $I_{u_1}^{(l)}$ we know that the intervals $f^{(i)}(I_{u_1}^{(l)})$ do not cross $0$ for $0\le i\le l$, and since $|f'(x)|\ge \sqrt{2}$ for every $x\in [-1/2,1/2]\setminus 0$, it follows that $\Delta_1(i+1)\ge \sqrt{2} \Delta_1(i)$ for every $0\le i<l$. By combining this with 
\eqref{delta2bound1} and using the fact that $\frac{2}{3\sqrt{2}}<\frac{1}{2}$ it follows that
\begin{equation}\label{delta2bound2}
\Delta_2(i+1)\le \frac{1}{2} \Delta_2(i) + \frac{1}{2} \Delta_1(i+1) \text{ for }0\le i<l.
\end{equation}
From this by induction we can obtain that $\Delta_2(i)\le \Delta_1(i)$ for any $0\le i\le l$ (
by the initial assumption on $\v$, we have $\Delta_2(0)=0$, so this holds for $i=0$). Thus the difference in the second coordinate is upper bounded by the difference in the first one.

For $1\le i\le l+1$, let $\tau_i(\u):=T_i(\u)-T_{i-1}(\u)$ and define $\tau_i(\v)$ analogously.
Based on \eqref{timedelayeq}, the time delay that is created between the two paths can be bounded as
\begin{align*}|\tau_i(\u)-\tau_i(\v)|&\le \frac{2}{\lambda_1} \frac{|f^{(i-1)}(u_1)-f^{(i-1)}(v_1)|}{|f^{(i-1)}(u_1)|}\\
&\le \frac{2}{\lambda_1} |f^{(l)}(u_1)-f^{(l)}(v_1)| \frac{2^{-(l-(i-1))/2}}{|f^{(i-1)}(u_1)|},
\end{align*}
thus for any $1\le j\le l+1$, we have
\begin{equation}\label{Tidiffmaxeq}\max_{1\le i\le j}|T_i(\u)-T_i(\v)|\le \frac{2}{\lambda_1} |f^{(l)}(u_1)-f^{(l)}(v_1)| \sum_{i=0}^{j-1}\frac{2^{-(l-i)/2}}{|f^{(i)}(u_1)|}.
\end{equation}
Moreover, using Lemma \ref{maxdistancefirstretlemma} and the fact that $\Delta_2(i)\le
\Delta_1(i)$, for any $0\le j\le l$, we have
\begin{align*}
\sup_{0\le r\le \tau_{j+1}(\u)} \|\u(T_{j}(\u)+r)-\v(T_{j}(\v)+r)\|_{\infty}&\le C_1\left(\frac{\left |f^{(j)}(u_1)-f^{(j)}(v_1)\right|}{|f^{(j)}(u_1)|}\cdot 2\right)\\
&\le 2C_1 |f^{(l)}(u_1)-f^{(l)}(v_1)| \frac{2^{-(l-j)/2}}{|f^{(j)}(u_1)|}.
\end{align*}
The statement of the lemma now follows by \eqref{vmaxgeodefeq} and the triangle inequality  with 
$C^{\mathrm{ret}}:=\max\left(2C_1,\frac{2}{\lambda_1}\cdot \vmaxgeo\right)$.
\end{proof}

The following lemma lower bounds the distance of two paths at time points when they are close to $S$. 
\begin{lem}[Distance of two paths near $S$]\label{lemmadistpathnearS}
Let $T_l(\u)$ be the $l$th return time from $\u$ to $S$ (i.e. the $l$th smallest $t>0$ such that $\u(t)\in S$). Let 
\[h^{S}_{\max}:=\frac{1}{20 \vmaxgeo}\text{ and }C^{S}:=\frac{\left(\vmingeo\right)^2}{4\vmaxgeo(\vmaxgeo+\vmingeo)}\cdot \exp\left(-h^{S}_{\max}\cdot \lambda_2 \right).\]
Then for any $l\ge 1$, $k\ge 1$ such that $t_k\in [T_l(\u),T_l(\u)+h^{S}_{\max}]$, for any $\v\in \tU(\u,k)$, we have 
\[\|\v(t_k)-\u(t_k)\|_{\infty}\ge C^{S} \left|f^{(l)}(O_1(\u))-f^{(l)}(O_1(\v))\right|.\]
\end{lem}
\begin{proof}
As in the proof of \eqref{eqpathdistancebound}, using Gr\"{o}nwall's inequality, and the fact that $0<\lambda_3<\lambda_1<\lambda_2$, one can show that for any $t\ge 0$, $\v,\w\in \R^3$,
\begin{equation}\label{geoGeq}
\|\v-\w\|_{\infty}\cdot \exp(-\lambda_2 t)\le \|\Psi_t^{\mathrm{lin}}(\v)-\Psi_t^{\mathrm{lin}}(\w)\|_{\infty}\le 
\|\v-\w\|_{\infty}\cdot \exp(\lambda_2 t).
\end{equation}
Based on the definition \eqref{vmaxgeodefeq}, we can see that for any $\v\in \tU(\u,k)$, we have
\[|f^{(l)}(O_1(\u))-f^{(l)}(O_1(\v))|\le \|\u(T_l(\u))-\v(T_l(\u))\|_{\infty}+ |T_l(\u)-T_l(\v)|\cdot \vmaxgeo.\]
From the definition of $\tU(\u,k)$ and \eqref{vmingeodefeq}, it follows that $|T_l(\u)-T_l(\v)|\le \|\u(T_l(\u))-\v(T_l(\u))\|_{\infty}\cdot \vmaxgeo$, thus 
\begin{equation}
\|\u(T_l(\u))-\v(T_l(\u))\|_{\infty}\ge |f^{(l)}(O_1(\u))-f^{(l)}(O_1(\v))|\cdot \frac{\vmingeo}{\vmingeo+\vmaxgeo}.
\end{equation}
Let $\rho(\u,\v):=|f^{(l)}(O_1(\u))-f^{(l)}(O_1(\v))|\cdot \frac{\vmingeo}{\vmingeo+\vmaxgeo}$. If $v_3(T_l(\u))\le 1$ (thus $\v(T_l(\u))$ is on $S$ or below $S$), then by \eqref{geoGeq}, for every $t_k\in [T_l(\u),T_l(\u)+h^{S}_{\max}]$, we have
\begin{equation}\label{rhouveq1}
\|\u(t_k)-\v(t_k)\|_{\infty}\ge \exp\left(-\lambda_2 h^{S}_{\max}\right) \rho(\u,\v).
\end{equation}
For $T_{l}(\u)\le t\le T_{l}(\u)+\frac{\rho(\u,\v)}{4\vmaxgeo}$, for 
\begin{equation}\label{rhouveq2}
\|\u(t)-\v(t)\|_{\infty}\ge \frac{\rho(\u,\v)}{2}.
\end{equation}
If $\frac{\rho(\u,\v)}{4\vmaxgeo}<h^{S}_{\max}$, and $v_3\left(T_l(\u)+\frac{\rho(\u,\v)}{4\vmaxgeo}\right)\le 1$, then by \eqref{geoGeq}, for every $t_k\in [T_l(\u),T_l(\u)+h^{S}_{\max}]$, we have
\begin{equation}\label{rhouveq3}
\|\u(t_k)-\v(t_k)\|_{\infty}\ge \exp\left(-\lambda_2 h^{S}_{\max}\right) \frac{\rho(\u,\v)}{2}.
\end{equation}
Finally, if $\frac{\rho(\u,\v)}{4\vmaxgeo}<h^{S}_{\max}$, and $v_3\left(T_l(\u)+\frac{\rho(\u,\v)}{4\vmaxgeo}\right)> 1$, then $u_3\left(t\right)\le 1-\rho(\u,\v) \cdot \frac{\vmingeo}{4\vmaxgeo}$ for 
$t\in [T_l(\u)+\frac{\rho(\u,\v)}{4\vmaxgeo},T_l(\u)+h^{S}_{\max}]$, and thus by \eqref{geoGeq}, for every $t_k\in [T_l(\u),T_l(\u)+h^{S}_{\max}]$, we have
\begin{equation}\label{rhouveq4}
\|\u(t_k)-\v(t_k)\|_{\infty}\ge \exp\left(-\lambda_2 h^{S}_{\max}\right) \rho(\u,\v) \cdot \frac{\vmingeo}{4\vmaxgeo}.
\end{equation}
The claim now follows from inequalities \eqref{rhouveq1}, \eqref{rhouveq2}, \eqref{rhouveq3} and \eqref{rhouveq4}.
\end{proof}

The following lemma bounds the differences $\sum_{i=0}^k \|\v(t_i)-\u(t_i)\|_{\infty}$ for $\v\in\tU(\u,k)$.
\begin{lem}\label{sumdiffbndlemma}
Let $D_l(x):=\sum_{i=0}^{l}\frac{2^{-(l-i)/4}}{\left(d(f^{(i)}( x ))\right)^2}$  for any $l\in \N$, $x\in D^{(l)}_f$. Suppose that $h\le \frac{3}{2}\pi$. Then there is a constant $C^{\mathrm{sum}}<\infty$ such that for every $l\ge 1$, every $k\in \N$ such that $t_k\in [T_l(\u),T_l(\u)+h_{\max}^S]$, every $\v\in \tU(\u,k)$, we have
\begin{equation}
\sum_{i=0}^{k} \|\v(t_i)-\u(t_i)\|_{\infty}\le C_{\tilde{U}}(\u,k) \|\v(t_k)-\u(t_k)\|_{\infty},
\end{equation}
where
\begin{equation}\label{CUukdefeq}
C_{\tilde{U}}(\u,k):=\frac{C^{\mathrm{sum}}}{h}\cdot D_l(O_1(\u)).
\end{equation}
\end{lem}
\begin{proof}
First note that based on the assumptions, it follows from Lemma \ref{lemmadistpathnearS} that for every $\v\in\tU(\u,k)$, we have
\begin{equation}\label{fldistboundeq1}
\left|f^{(l)}(O_1(\u))-f^{(l)}(O_1(\v))\right|\le \frac{1}{C^{S}}\|\v(t_k)-\u(t_k)\|_{\infty}.
\end{equation}
From Lemma \ref{lemmamaxdistanceuntillthreturn}, we know that for any $0\le j\le l$, $t_i\in [T_{j}(\u),T_{j+1}(\u)]$, we have
\begin{equation}
\|\u(t_i)-\v(t_i)\|_{\infty}\le C^{\mathrm{ret}} |f^{(l)}(O_1(\u))-f^{(l)}(O_1(\v))| \sum_{i=0}^{j} \frac{2^{-(l-i)/2}}{|f^{(i)}(O_1(\u))|}.
\end{equation}
By the assumption $h\le \frac{3}{2}\pi$, it is easy to see that there are at most $\frac{2 \tau_{j+1}(\u)}{h}$ such indices $i$. From \eqref{eqreturntime}, and the fact that $\log(x)\le \frac{x}{2}$ for $x\ge 2$, we can see that 
\[\tau_{j+1}(\u)=\frac{1}{\lambda_1}\log\left(\frac{1}{|f^{(j)}(O_1(\u))|}\right)+\frac{3}{2}\pi\le \frac{1}{2\lambda_1 |f^{(j)}(O_1(\u))|}+\frac{3}{2}\pi\le \frac{1+3\pi \lambda_1}{2\lambda_1}\cdot \frac{1}{|f^{(j)}(O_1(\u))|}.\]
Let $C_{2}:=\frac{2}{h}\cdot \frac{1+3\pi \lambda_1}{2\lambda_1}\cdot \frac{C^{\mathrm{ret}}}{C^{S}}$. By summing up, and using \eqref{fldistboundeq1}, we obtain that
\begin{align*}
&\sum_{i=0}^{k} \|\v(t_i)-\u(t_i)\|_{\infty}\le C_2 \cdot \|\v(t_k)-\u(t_k)\|_{\infty} \cdot \sum_{j=0}^{l} \sum_{i=0}^{j} \frac{2^{-(l-i)/2}}{|f^{(j)}(O_1(\u))| |f^{(i)}(O_1(\u))|}\\
&\le \frac{C_2}{2} \cdot \|\v(t_k)-\u(t_k)\|_{\infty} \cdot \sum_{j=0}^{l} \sum_{i=0}^{j} 2^{-(l-i)/2}\cdot \left(\frac{1}{(f^{(j)}(O_1(\u)))^2}+ \frac{1}{(f^{(i)}(O_1(\u)))^2}\right)\\
&\le \frac{C_2}{2} \cdot \|\v(t_k)-\u(t_k)\|_{\infty} \cdot \sum_{j=0}^{l} 2^{-(l-j)/2}\cdot \frac{2+(l-j+1)}{(f^{(j)}(O_1(\u)))^2}.
\end{align*}
Now using the fact that $2^{-(l-j)/2}\cdot (l-j+3)<4 \cdot  2^{-(l-j)/4}$, we obtain that
\begin{equation}
\sum_{i=0}^{k} \|\v(t_i)-\u(t_i)\|_{\infty}\le 2C_2 \cdot \|\v(t_k)-\u(t_k)\|_{\infty} \cdot \sum_{j=0}^{l} \frac{2^{-(l-j)/4}}{(f^{(j)}(O_1(\u)))^2},
\end{equation}
thus the result follows with $C^{\mathrm{sum}}:= \frac{2(1+3\pi \lambda_1)}{\lambda_1}\cdot \frac{C^{\mathrm{ret}}}{C^{S}}$.
\end{proof}

The next lemma characterises the set $\{\|\v(t_k)-\u(t_k)\|: \v\in \tU(\u,k)\}$.
\begin{lem}\label{absvudistlemma}
For every $l\ge 1$, every $k\in \N$ such that $t_k\in [T_l(\u),T_l(\u)+h_{\max}^S]$,
we have 
$\{\|\v(t_k)-\u(t_k)\|_{\infty}: \v\in \tU(\u, k)\}=[0,\tilde{d}_{\max}(\u,k)]$, with
\begin{equation}
\tilde{d}_{\max}(\u,k)\ge \frac{4}{5}\cdot \frac{1}{D_l(O_1(\u))}.
\end{equation}
\end{lem}
\begin{proof}
From the definition of $\tU(\u, k)$, it follows that it is a continuous curve in $\R^3$, and thus 
$\{\|\v(t_k)-\u(t_k)\|_{\infty}: \v\in \tU(\u, k)\}=[0,\tilde{d}_{\max}(\u,k)]$ for some $\tilde{d}_{\max}(\u,k)\ge 0$. Based on the definition of the intervals $I^{(l)}_{O_1(\u)}$ and \eqref{deltakdef}, we know that 
\[f^{(l)}(I^{(l)}_{O_1(\u)})=[f^{(l)}(O_1(\u))-\delta^{(l)}_{O_1(\u)}/2,f^{(l)}(O_1(\u))+\delta^{(l)}_{O_1(\u)}/2],\]
thus by Lemma \ref{lemmadistpathnearS}, it follows that there must exist a point $\v\in \tU(\u,k)$ such that
\[\|\v(t_k)-\u(t_k)\|_{\infty}\ge \min \left(\frac{1}{20}, C_S\cdot \frac{\delta^{(l)}_{O_1(\u)}}{2}\right) \ge \frac{1}{5} C_S \delta^{(l)}_{O_1(\u)}.\]
Since $\delta^{(l)}_{O_1(\u)}\ge \frac{4}{D_l(O_1(\u))}$, the result follows.
\end{proof}

The following proposition is a consequence of Proposition \ref{Prop22}.
\begin{prop}\label{propempav}
For $\mu_f$-almost surely every initial point $x\in [-1/2,1/2]$, for any function $r: [-1/2,1/2]\to \R$ such that 
$m(|r|)=\int_{x=-1/2}^{1/2}|r(x)| dx<\infty$, we have
\[\lim_{n\to \infty}\frac{1}{n}\sum_{i=0}^{n-1} r\left(f^{(i)}(x)\right) = \mu_f(r)<\infty.\]
\end{prop}
\begin{proof}
Since $\frac{d \mu_f}{d m}$ is bounded, we have $\mu_f(|r|)<\infty$. From Proposition \ref{Prop22} we know that $\mu_f$ is ergodic for the map $f$, so the stated results follows by Birkhoff's ergodic theorem (see, e.g.~Corollary 3.8 of \cite{Fundamentalsofmeasurabledynamics}).
\end{proof}

Now we are ready to prove the existence of the anti-leaf sets.

\begin{proof}[Proof of Theorem \ref{thmgeoantileafsets}]
By the definition of $D_l(x)$, we have
\[(D_l(x))^{1/4}\le \sum_{i=0}^{l}\frac{2^{-(l-i)/16}}{\left(d(f^{(i)}( x )\right)^{1/2}},\]
and by summing up in $l$, we obtain that
\[\frac{1}{l+1}\sum_{i=0}^{l}D_l^{1/4}(x)\le \frac{32}{l+1}\sum_{i=0}^{l}\frac{1}{\left(d(f^{(i)}( x )\right)^{1/2}}.\]
Since the function $r(x)= d(x)^{-\frac{1}{2}}$ is integrable on the interval $[-1/2,1/2]$,
by Proposition \ref{propempav}, we obtain that 
\begin{equation}
\limsup_{l\to \infty}\frac{1}{l+1}\sum_{i=0}^{l}D_l^{1/4}(x)\le 32 \mu_f(r)<\infty.
\end{equation}
for $\mu_f$-almost every $x\in [-1/2,1/2]$. Therefore, for every such $x$, we have an infinite sequence of indices $l_1, l_2,\ldots$ satisfying that for every $j\ge 1$, $D_{l_j}^{1/4}(x)\le 40\mu_f(r)$, and thus $D_{l_j}(x)\le 2560000 (\mu_f(r))^4$. 
Set $h_{\max}:=h^{S}_{\max}=\frac{1}{20 \vmaxgeo}$, then for every such $x$, every $\u\in \Lambdageo$ such that $O_1(\u)=x$, every $j\ge 1$, every there exists an index $i_j$ such that $t_{i_j}\in [T_{l_j}(\u),T_{l_j}(\u)+h^{S}_{\max}]$, and therefore the results of the theorem follow from Lemmas \ref{sumdiffbndlemma} and \ref{absvudistlemma} with
\[\tilde{d}_{\max}(\u):=\frac{1}{3200000 (\mu_f(r))^4} \quad \text{ and } \quad C_{\tU}(\u):=\frac{2560000 (\mu_f(r))^4 C^{\mathrm{sum}}  }{h}.\qedhere\]
\end{proof}

\subsection{Characterisation of the limit of the support of the smoother} \label{secproofthmcharacterizationsmoother}
\begin{proof}[Proof of Theorem \ref{thmcharacterizationsmoother}]
Let 
\[S_{\infty}:=\cap_{k\in \N}\,\supp(\musm(\cdot|Y_0,\ldots,Y_k)).\]
Then as we have explained after equation \eqref{eqSkintersection}, \eqref{eqcharsuppsm1} is equivalent to the fact that $S_{\infty}\subset \ol{U(\u,\epsilon)}$ (the closure of $U(\u,\epsilon)$) almost surely in the observations.

From \eqref{eqnotinWuepstmax} of Assumption \ref{assunifchar}, we know that for any $S_{\infty}\subset W(\u,\epsilon,t_{\max}(\u,\epsilon))$. Therefore we only need to check that the points 
$\v\in W(\u,\epsilon,t_{\max}(\u,\epsilon))\setminus \ol{U(\u,\epsilon)}$ are not in $S_{\infty}$. For such points, we define $\Delta(\v)$ as the value of $t$ in \eqref{Wuepstmaxdefeq}. For $0<s\le t_{\max}(\u,\epsilon)$, we define the sets
\begin{align*}
W_+(\u,\epsilon,s)&:=\{\v\in W(\u,\epsilon,t_{\max}(\u,\epsilon)), \Delta(\v)\ge s\}, \text{ and }\\
W_-(\u,\epsilon,s)&:=\{\v\in W(\u,\epsilon,t_{\max}(\u,\epsilon)), \Delta(\v)\le -s\},
\end{align*}
called the restrictions of the time-shifted $2\epsilon$-cropped leaf sets. Let $i_1, i_2,\ldots $ be the set of indices satisfying \eqref{eqdotutivmin}, and suppose that $k$ is sufficiently large such that $\rho_{i_k}(\u,\epsilon)<\frac{s v_{\min}(\u)}{3}$ and $t_{i_k}\ge t_{\max}(\u,\epsilon)$. Let $\w \in U(\u,\epsilon)$, then for every $t\in [-t_{\max}(\u,\epsilon), t_{\max}(\u,\epsilon)]$, using the assumption that $t_{\max}(\u,\epsilon)\in (0, \frac{v_{\min}(\u)}{6a_{\max} v_{\max}})$, we have
\begin{equation}\label{wttmaxeq}
\left\|\w\left(t_{i_k}+t\right)-\u\left(t_{i_k}\right)\right\|\le \rho_{i_k}(\u,\epsilon)+t_{\max}(\u,\epsilon)v_{\max}<\frac{v_{\min}(\u)}{3a_{\max}},
\end{equation}
Let $\frac{d}{dt}u_j(t_{i_k})$ be a component of $\frac{d}{dt}\u(t_{i_k})$ with the largest magnitude, then $|\frac{d}{dt}u_j(t_{i_k})|\ge v_{\min}$. Assume without loss of generality that $\frac{d}{dt}u_j(t_{i_k})>0$ (the negative case can be dealt with in the same way). Then by \eqref{eqamax} and \eqref{wttmaxeq}, we have $\frac{d}{dt}w_j\left(t_{i_k}+t\right)\ge v_{\min}(\u)-\frac{v_{\min}(\u)}{3a_{\max}}\cdot a_{\max}=\frac{2}{3}v_{\min}(\u)$ for every $\w\in U(\u,\epsilon)$ and $t\in [-t_{\max}(\u,\epsilon), t_{\max}(\u,\epsilon)]$. By using this property, we can show that $\inf_{\v\in W_+(\u,\epsilon,s)}v_j(t_{i_k})-u_j(t_{i_k})\ge \frac{s v_{\min}(\u)}{3}$. This means that if the $j$th component of the observation error at time $t_{i_k}$, denoted by $Z_{t_{i_k}}^j$, is less than $-\epsilon+\frac{s v_{\min}(\u)}{3}$, then none of the elements in $W_+(\u,\epsilon,s)$ is included in the limiting set $S_{\infty}$. Since $Z_{t_{i_k}}^j$ is assumed to be uniformly distributed on $[-\epsilon,\epsilon]$, we have $\PP\left(Z_{t_{i_k}}^j<-\epsilon+\frac{s v_{\min}(\u)}{3}\right)=\min\left(\frac{s v_{\min}(\u)}{6\epsilon},1\right)>0$. Since there are infinitely many such indices $i_k$ where this holds, so we have 
$W_+(\u,\epsilon,s)\cap S_{\infty}=\emptyset$ for any $s>0$ almost surely, and with an analogous argument, we have 
$W_-(\u,\epsilon,s)\cap S_{\infty}=\emptyset$ almost surely too. The first statement of the theorem, \eqref{eqcharsuppsm1} now follows by the union bound, since we can write $W(\u,\epsilon,t_{\max}(\u,\epsilon))\setminus  \ol{U(\u,\epsilon)}$ as a countable union 
\[W(\u,\epsilon,t_{\max}(\u,\epsilon))\setminus \ol{U(\u,\epsilon)}=\cup_{i\ge 1} \left(W_+(\u,\epsilon,t_{\max}(\u,\epsilon)/i) \cup W_-(\u,\epsilon,t_{\max}(\u,\epsilon)/i)\right),\]
and therefore almost surely, none of the points in $W(\u,\epsilon,t_{\max}(\u,\epsilon))\setminus \ol{U(\u,\epsilon)}$ are included in the limiting set $S_{\infty}$. The final statement of the theorem follows from the definition of $U(\u,\epsilon)$ and \eqref{probvinsupporteq}.
\end{proof}

\subsection{Assumptions on derivatives imply assumptions for upper bounds}\label{secassderproof}
\begin{proof}[Proof of Proposition \ref{Propcheckassumptions}]
From inequality \eqref{uderboundeq}, we can see that the Taylor expansion
\[\H \v(t)=\sum_{i=0}^{\infty} \frac{\H \D^i \v \cdot t^i}{i!}\]
is valid for times $0\le t< \Cder^{-1}$. Based on this expansion, assuming that $t_i< \Cder^{-1}$, the $i$th derivatives  $\H \D^i \v$ can be approximated by a finite difference formula (see \cite{Finitediffformulas}) depending on the values $\Phi_0(\v),\ldots, \Phi_{t_i}(\v)$, with error of $O(h)$. This approximation will be denoted by
\begin{equation}\label{hatPhidef}
\hat{\Phi}^{(i)}(\v):=\frac{\sum_{l=0}^{i}a_l^{(i)} \Phi_{t_l}(\v)}{h^{i}},
\end{equation}
where $a_l^{(i)}$ are some absolute constants independent of $h$ and $\v$.

By Taylor's expansion of the terms $\Phi_{t_l}(\v)$ around time point 0, for $t_l< \Cder^{-1}$, we have
\begin{align*}
\Phi_{t_l}(\v)&=\sum_{m=0}^{\infty}\H \D^m \v \cdot \frac{h^m l^m}{m!},\text{ and thus }\\
\hat{\Phi}^{(i)}(\v)&=\frac{1}{h^i}\cdot \sum_{l=0}^{\infty}a_l^{(i)} \sum_{m=0}^{\infty}\H \D^m \v\cdot \frac{h^m l^m}{m!}=\frac{1}{h^i}\sum_{m=0}^{\infty}h^m b_m^{(i)}\H \D^m \v,
\end{align*}
with $b_{m}^{(i)}:=\frac{1}{m!}\cdot \sum_{l=0}^{i} a_l^{(i)}l^m$. Due to the particular choice of the constants $a_l^{(i)}$, we have $b_{m}^{(i)}=0$ for $0\le m<i$ and $b_{m}^{(i)}=1$ for $m=i$. Based on this, we can write the difference between the approximation \eqref{hatPhidef} and the derivative explicitly as
\[\hat{\Phi}^{(i)}(\v)-\frac{d^i}{d t^i}\Phi_0(\v)=h  \left(\sum_{m=i+1}^{\infty}h^{m-i-1} \cdot b_{m}^{(i)}   \cdot\H \D^m \v \right).\]
Let us denote $\tilde{\Phi}^{(i)}(v,h):=\sum_{m=i+1}^{\infty}h^{m-i-1} \cdot b_{m}^{(i)}   \cdot\H \D^m \v$. Let $\ol{a}:=\max_{0\le i\le j, 0\le l\le i} |a_l^{(i)}|$, then we have $|b_{m}^{(i)}|\le \ol{a} \frac{i+1}{m!}\cdot i^m$. Using inequality \eqref{ugradboundeq}, one can show that for any $\v \in \BR$, for $h<\frac{1}{2 j \CJ}$, we have
\begin{align*}
\|\mtx{J}_{\v}\tilde{\Phi}^{(i)}(v,h)\|&\le \sum_{m=i+1}^{\infty} h^{m-i-1} |b_{m}^{(i)}| \CJ^{m} m!
\le \frac{\ol{a} (i+1)}{h^{i+1}} \sum_{m=i+1}^{\infty} (ih C_J)^{m}\\
&=\frac{\ol{a}(i+1)(iC_J)^{i+1}}{1-ihC_J}\le 2\ol{a}(i+1)(iC_J)^{i+1}.
\end{align*}
Denote $C_{\mathrm{Lip}}:=\max_{0\le i\le j}2\ol{a}(i+1) (i\CJ)^{i+1}$, then for every $0\le i\le j$, $h<\frac{1}{2 j \CJ}$, the functions $\tilde{\Phi}^{(i)}(v,h)$ are $C_{\mathrm{Lip}}$ - Lipschitz in $\v$ with respect to the $\|\cdot \|$ norm. Thus for every $0\le i\le j$, $h<\frac{1}{2 j \CJ}$, we have
\begin{equation}\label{eqlip}\left\|\left[\hat{\Phi}^{(i)}(\v)-\hat{\Phi}^{(i)}(\u)\right] -\left[\frac{d^i}{d t^i}\Phi_0(\v)-\frac{d^i}{d t^i}\Phi_0(\u)\right] \right\|\le 
h\cdot C_{\mathrm{Lip}} \cdot \|\u-\v\|.
\end{equation}
Let $\mtx{M}_j(\u):=\sum_{i=0}^{j}\sum_{k=1}^{d_o}\left(\grad\H \D^i \u\cdot \left(\grad \H \D^i \u\right)'\right)$, then by \eqref{eqlambdaminder}, we have for any $v\in \R^d$,
\[(\v-\u)' \mtx{M}_j(\u) (\v-\u)\ge \lambda_{\min}(\mtx{M}_j(\u)) \|\v-\u\|^2.\]
Based on this and the second order Taylor expansion, one can show that for $\|\v-\u\|$ sufficiently small, we have
\begin{equation}
\sum_{i=0}^{j} \left\|\H \D^i \u-\H \D^i \v\right\|^2\ge \frac{1}{2}\lambda_{\min}(\mtx{M}_j(\u)) \|\v-\u\|^2,
\end{equation}
and from the compactness of $\BR$ and the uniqueness condition \eqref{eqphitideruv}, it follows that
there is a constant $c_j(\u)>0$ such that
\begin{equation}
\sum_{i=0}^{j} \left\|\H \D^i \u-\H \D^i \v\right\|^2\ge c_j(\u) \|\v-\u\|^2 \text{ for every } \v\in \BR.
\end{equation}
Using this and \eqref{eqlip} we obtain for $h$ sufficiently small, we have that for every $\v\in \BR$,
\[\max_{0\le i\le j}\left\|\hat{\Phi}^{(i)}(\v)-\hat{\Phi}^{(i)}(\u)\right\|\ge \frac{\sqrt{c_j(\u)}}{2\sqrt{j+1}}\cdot \|\u-\v\|.\]
From the definition \eqref{hatPhidef}, we have
\[\max_{0\le i\le j}\left\|\hat{\Phi}^{(i)}(\v)-\hat{\Phi}^{(i)}(\u)\right\|\le \left[\max_{0\le i\le j, 0\le l\le i}
\frac{|a_l^{(i)}|}{h^i}\right]\cdot \max_{0\le i\le j} \|\Phi_{t_i}(\u)- \Phi_{t_i}(\v)\|,\]
and the conclusion follows.
\end{proof}

\bibliographystyle{plain}
\bibliography{References}

\end{document}